\documentclass[11pt,twoside,english,british]{article}
\usepackage[T1]{fontenc}
\usepackage[latin9]{inputenc}
\usepackage{geometry}
\usepackage{fancyhdr}
\usepackage{color}
\usepackage{array}
\usepackage{verbatim}
\usepackage{longtable}
\usepackage{refstyle}
\usepackage{mathrsfs}
\usepackage{mathtools}
\usepackage{amsthm}
\usepackage{amsmath}
\usepackage{amssymb}
\usepackage{setspace}
\usepackage{esint}
\usepackage[authoryear]{natbib}
\usepackage{xargs}[2008/03/08]
\makeatletter

\AtBeginDocument{\providecommand\eqref[1]{\ref{eq:#1}}}
\AtBeginDocument{\providecommand\propref[1]{\ref{prop:#1}}}
\AtBeginDocument{\providecommand\thmref[1]{\ref{thm:#1}}}
\AtBeginDocument{\providecommand\appref[1]{\ref{app:#1}}}
\AtBeginDocument{\providecommand\enuref[1]{\ref{enu:#1}}}
\AtBeginDocument{\providecommand\lemref[1]{\ref{lem:#1}}}
\AtBeginDocument{\providecommand\assref[1]{\ref{ass:#1}}}
\AtBeginDocument{\providecommand\secref[1]{\ref{sec:#1}}}
\AtBeginDocument{\providecommand\remref[1]{\ref{rem:#1}}}
\AtBeginDocument{\providecommand\subref[1]{\ref{sub:#1}}}
\providecommand{\tabularnewline}{\\}
\RS@ifundefined{subref}
  {\def\RSsubtxt{section~}\newref{sub}{name = \RSsubtxt}}
  {}
\RS@ifundefined{thmref}
  {\def\RSthmtxt{theorem~}\newref{thm}{name = \RSthmtxt}}
  {}
\RS@ifundefined{lemref}
  {\def\RSlemtxt{lemma~}\newref{lem}{name = \RSlemtxt}}
  {}

\usepackage{enumitem}		
\numberwithin{equation}{section}
  \theoremstyle{plain}
  \newtheorem{assumption}{\protect\assumptionname}
  \theoremstyle{remark}
  \newtheorem{rem}{\protect\remarkname}[section]
  \theoremstyle{plain}
  \newtheorem{prop}{\protect\propositionname}[section]
  \theoremstyle{plain}
  \newtheorem{thm}{\protect\theoremname}[section]
  \theoremstyle{plain}
  \newtheorem{lem}{\protect\lemmaname}[section]

\setlist[enumerate,1]{label=\textnormal{(\roman*)}}
\setlist[enumerate,2]{label=\textnormal{(\alph*)}}

\makeatletter
  \@ifpackageloaded{refstyle}{}{\usepackage{refstyle}}
\makeatother

\newref{thm}{name = Theorem~}
\newref{prop}{name = Proposition~}
\newref{lem}{name = Lemma~}
\newref{corr}{name = Corollary~}
\newref{ass}{name = Assumption~}
\newref{exa}{name = Example~}
\newref{rem}{name = Remark~}
\newref{def}{name = Definition~}

\newref{eq}{refcmd = (\refeq{#1})}

\newref{sec}{name = Section~}
\newref{sub}{name = Section~}
\newref{app}{name = Section~}

\newref{fig}{name = Figure~}
\newref{tbl}{name = Table~}

\newref{enu}{name = {}}

\usepackage[user]{zref}

\makeatletter
  \def\olabel#1#2{\begingroup
     \def\@currentlabel{#2}%
     \zlabel{#1}\endgroup
  }
\makeatother

\newcommand{\objlabel}[2]{\olabel{obj:#2}{#1}#1}%
\newcommand{\objref}[1]{\zref{obj:#1}}

\makeatletter
  \@ifpackageloaded{mathrsfs}{}{\usepackage{mathrsfs}}
\makeatother

\date{}

\usepackage{nextpage}

\allowdisplaybreaks[1]

\numberwithin{equation}{section}

\usepackage{thmtools}
\usepackage{thm-restate}

\newcounter{sectemp}
\newcounter{eqtemp}

\newcommand{\restoreposition}[1]{
  \setcounter{section}{\value{sectemp}}
  \setcounter{equation}{\value{eqtemp}}
}

\usepackage[bitstream-charter,cal=cmcal]{mathdesign}

\usepackage{mleftright}
\renewcommand{\left}{\mleft}
\renewcommand{\right}{\mright}

\providecommand{\LyX}{L\kern-.0467em\lower.25em\hbox{Y}\kern-.025emX\@} 

\mathtoolsset{showonlyrefs=true,showmanualtags=true}

\makeatletter
\newcommand*{\getlength}[1]{\strip@pt#1}
\makeatother

\newlength{\tempjot}
\newcommand{\setjot}[1]{\setlength{\tempjot}{\getlength{\jot}pt}\setlength{\jot}{#1pt}}
\newcommand{\restorejot}{\setlength{\jot}{\getlength{\tempjot}pt}\setlength{\tempjot}{100pt}}

\makeatletter
\newcommand \cellfill {\leavevmode \leaders \hb@xt@ .44em{\hss .\hss }\hfill \kern \z@}
\makeatother

\makeatother

\usepackage{babel}
  \addto\captionsbritish{\renewcommand{\assumptionname}{Assumption}}
  \addto\captionsbritish{\renewcommand{\lemmaname}{Lemma}}
  \addto\captionsbritish{\renewcommand{\propositionname}{Proposition}}
  \addto\captionsbritish{\renewcommand{\remarkname}{Remark}}
  \addto\captionsenglish{\renewcommand{\assumptionname}{Assumption}}
  \addto\captionsenglish{\renewcommand{\lemmaname}{Lemma}}
  \addto\captionsenglish{\renewcommand{\propositionname}{Proposition}}
  \addto\captionsenglish{\renewcommand{\remarkname}{Remark}}
  \providecommand{\assumptionname}{Assumption}
  \providecommand{\lemmaname}{Lemma}
  \providecommand{\propositionname}{Proposition}
  \providecommand{\remarkname}{Remark}
\addto\captionsbritish{\renewcommand{\theoremname}{Theorem}}
\addto\captionsenglish{\renewcommand{\theoremname}{Theorem}}
\providecommand{\theoremname}{Theorem}

\begin{document}
\selectlanguage{english}%


\global\long\def\uwrite#1#2{\underset{#2}{\underbrace{#1}} }

\global\long\def\blw#1{\ensuremath{\underline{#1}}}

\global\long\def\abv#1{\ensuremath{\overline{#1}}}

\global\long\def\vect#1{\mathbf{#1}}


\global\long\def\smlseq#1{\{#1\} }

\global\long\def\seq#1{\left\{  #1\right\}  }

\global\long\def\smlsetof#1#2{\{#1\mid#2\} }

\global\long\def\setof#1#2{\left\{  #1\mid#2\right\}  }


\global\long\def\goesto{\ensuremath{\rightarrow}}

\global\long\def\ngoesto{\ensuremath{\nrightarrow}}

\global\long\def\uto{\ensuremath{\uparrow}}

\global\long\def\dto{\ensuremath{\downarrow}}

\global\long\def\uuto{\ensuremath{\upuparrows}}

\global\long\def\ddto{\ensuremath{\downdownarrows}}

\global\long\def\ulrto{\ensuremath{\nearrow}}

\global\long\def\dlrto{\ensuremath{\searrow}}


\global\long\def\setmap{\ensuremath{\rightarrow}}

\global\long\def\elmap{\ensuremath{\mapsto}}

\global\long\def\compose{\ensuremath{\circ}}

\global\long\def\cont{C}

\global\long\def\cadlag{D}

\global\long\def\Ellp#1{\ensuremath{L^{#1}}}


\global\long\def\naturals{\ensuremath{\mathbb{N}}}

\global\long\def\reals{\mathbb{R}}

\global\long\def\complex{\mathbb{C}}

\global\long\def\rationals{\mathbb{Q}}

\global\long\def\integers{\mathbb{Z}}


\global\long\def\abs#1{\ensuremath{\left|#1\right|}}

\global\long\def\smlabs#1{\ensuremath{\lvert#1\rvert}}
 \global\long\def\bigabs#1{\ensuremath{\bigl|#1\bigr|}}
 \global\long\def\Bigabs#1{\ensuremath{\Bigl|#1\Bigr|}}
 \global\long\def\biggabs#1{\ensuremath{\biggl|#1\biggr|}}

\global\long\def\norm#1{\ensuremath{\left\Vert #1\right\Vert }}

\global\long\def\smlnorm#1{\ensuremath{\lVert#1\rVert}}
 \global\long\def\bignorm#1{\ensuremath{\bigl\|#1\bigr\|}}
 \global\long\def\Bignorm#1{\ensuremath{\Bigl\|#1\Bigr\|}}
 \global\long\def\biggnorm#1{\ensuremath{\biggl\|#1\biggr\|}}


\global\long\def\Union{\ensuremath{\bigcup}}

\global\long\def\Intsect{\ensuremath{\bigcap}}

\global\long\def\union{\ensuremath{\cup}}

\global\long\def\intsect{\ensuremath{\cap}}

\global\long\def\pset{\ensuremath{\mathcal{P}}}

\global\long\def\clsr#1{\ensuremath{\overline{#1}}}

\global\long\def\symd{\ensuremath{\Delta}}

\global\long\def\intr{\operatorname{int}}

\global\long\def\cprod{\otimes}

\global\long\def\Cprod{\bigotimes}


\global\long\def\smlinprd#1#2{\ensuremath{\langle#1,#2\rangle}}

\global\long\def\inprd#1#2{\ensuremath{\left\langle #1,#2\right\rangle }}

\global\long\def\orthog{\ensuremath{\perp}}

\global\long\def\dirsum{\ensuremath{\oplus}}


\global\long\def\spn{\operatorname{sp}}

\global\long\def\rank{\operatorname{rk}}

\global\long\def\proj{\operatorname{proj}}

\global\long\def\tr{\operatorname{tr}}

\global\long\def\diag{\operatorname{diag}}


\global\long\def\smpl{\ensuremath{\Omega}}

\global\long\def\elsmp{\ensuremath{\omega}}

\global\long\def\sigf#1{\mathcal{#1}}

\global\long\def\sigfield{\ensuremath{\mathcal{F}}}
\global\long\def\sigfieldg{\ensuremath{\mathcal{G}}}

\global\long\def\flt#1{\mathcal{#1}}

\global\long\def\filt{\mathcal{F}}
\global\long\def\filtg{\mathcal{G}}

\global\long\def\Borel{\ensuremath{\mathcal{B}}}

\global\long\def\cyl{\ensuremath{\mathcal{C}}}

\global\long\def\nulls{\ensuremath{\mathcal{N}}}

\global\long\def\gauss{\mathfrak{g}}

\global\long\def\leb{\mathfrak{m}}


\global\long\def\prob{P}

\global\long\def\Prob{\ensuremath{\mathbb{P}}}

\global\long\def\Probs{\mathcal{P}}

\global\long\def\PROBS{\mathcal{M}}

\global\long\def\expect{\mathbb{E}}

\global\long\def\probspc{\ensuremath{(\smpl,\filt,\Prob)}}


\global\long\def\iid{\ensuremath{\textnormal{i.i.d.}}}

\global\long\def\as{\ensuremath{\textnormal{a.s.}}}

\global\long\def\asp{\ensuremath{\textnormal{a.s.p.}}}

\global\long\def\io{\ensuremath{\ensuremath{\textnormal{i.o.}}}}

\newcommand\independent{\protect\mathpalette{\protect\independenT}{\perp}}
\def\independenT#1#2{\mathrel{\rlap{$#1#2$}\mkern2mu{#1#2}}}

\global\long\def\indep{\independent}

\global\long\def\distrib{\ensuremath{\sim}}

\global\long\def\distiid{\ensuremath{\sim_{\iid}}}

\global\long\def\asydist{\ensuremath{\overset{a}{\distrib}}}

\global\long\def\inprob{\ensuremath{\overset{p}{\goesto}}}

\global\long\def\inprobu#1{\ensuremath{\overset{#1}{\goesto}}}

\global\long\def\inas{\ensuremath{\overset{\as}{\goesto}}}

\global\long\def\eqas{=_{\as}}

\global\long\def\inLp#1{\ensuremath{\overset{\Ellp{#1}}{\goesto}}}

\global\long\def\indist{\ensuremath{\overset{d}{\goesto}}}

\global\long\def\eqdist{=_{d}}

\global\long\def\wkc{\ensuremath{\rightsquigarrow}}

\global\long\def\fdd{\ensuremath{\rightsquigarrow_{\mathrm{fdd}}}}

\global\long\def\wkcu#1{\overset{#1}{\ensuremath{\rightsquigarrow}}}

\global\long\def\plim{\operatorname*{plim}}


\global\long\def\var{\operatorname{var}}

\global\long\def\lrvar{\operatorname{lrvar}}

\global\long\def\cov{\operatorname{cov}}

\global\long\def\corr{\operatorname{corr}}

\global\long\def\bias{\operatorname{bias}}

\global\long\def\MSE{\operatorname{MSE}}

\global\long\def\med{\operatorname{med}}


\global\long\def\simple{\mathcal{R}}

\global\long\def\sring{\mathcal{A}}

\global\long\def\sproc{\mathcal{H}}

\global\long\def\Wiener{\ensuremath{\mathbb{W}}}

\global\long\def\sint{\bullet}

\global\long\def\cv#1{\left\langle #1\right\rangle }

\global\long\def\smlcv#1{\langle#1\rangle}

\global\long\def\qv#1{\left[#1\right]}

\global\long\def\smlqv#1{[#1]}


\global\long\def\trans{\ensuremath{\prime}}

\global\long\def\indic{\ensuremath{\mathbf{1}}}

\global\long\def\Lagr{\mathcal{L}}

\global\long\def\grad{\nabla}

\global\long\def\pmin{\ensuremath{\wedge}}
\global\long\def\Pmin{\ensuremath{\bigwedge}}

\global\long\def\pmax{\ensuremath{\vee}}
\global\long\def\Pmax{\ensuremath{\bigvee}}

\global\long\def\sgn{\operatorname{sgn}}

\global\long\def\argmin{\operatorname*{argmin}}

\global\long\def\argmax{\operatorname*{argmax}}

\global\long\def\Rp{\operatorname{Re}}

\global\long\def\Ip{\operatorname{Im}}

\global\long\def\deriv{\ensuremath{\mathrm{d}}}

\global\long\def\diffnspc{\ensuremath{\deriv}}

\global\long\def\diff{\ensuremath{\,\deriv}}

\global\long\def\i{\ensuremath{\mathrm{i}}}

\global\long\def\e{\mathrm{e}}

\global\long\def\sep{,\ }

\global\long\def\defeq{\coloneqq}

\global\long\def\eqdef{\eqqcolon}

\selectlanguage{british}%

\selectlanguage{english}%

\global\long\def\S{\mathcal{S}}

\global\long\def\M{\mathcal{M}}

\global\long\def\N{\mathcal{N}}

\global\long\def\V{\mathcal{V}}

\global\long\def\U{\mathcal{U}}

\global\long\def\R{\mathcal{R}}

\global\long\def\mg{\xi}

\global\long\def\minidx{k_{0}}

\global\long\def\minidxm{m_{0}}

\global\long\def\smltail#1{\smlnorm{#1}_{\omega}}

\global\long\def\alphnorm#1#2{\norm{#1}_{[#2]}}

\global\long\def\smlalph#1#2{\smlnorm{#1}_{[#2]}}

\global\long\def\smlmom#1#2{\smlnorm{#1}_{#2}}

\newcommandx\RV[1][usedefault, addprefix=\global, 1=]{\mathrm{RV}(#1)}

\global\long\def\BI{\mathrm{BI}}

\global\long\def\BInorm#1{\smlnorm{#1}_{\BI}}

\global\long\def\BIalph#1{\mathrm{BI}_{[#1]}}

\global\long\def\BILalph#1{\mathrm{BIL}_{[#1]}}

\global\long\def\BImom#1{\mathrm{BI}_{#1}}

\global\long\def\BILmom#1{\mathrm{BIL}_{#1}}

\global\long\def\BIz{\mathrm{BI_{0}}}

\global\long\def\BIc{\mathrm{BIC}}

\global\long\def\BIcz{\mathrm{BIC}_{0}}

\global\long\def\BIl{\mathrm{BIL}}

\global\long\def\BIlz{\mathrm{BIL}_{0}}

\global\long\def\Lip{\mathrm{Lip}}

\global\long\def\Reg{\mathrm{BI_{R}}}

\global\long\def\locest{\mathcal{L}_{n}}

\global\long\def\loctime{\mathcal{L}}

\global\long\def\lmest{\mathcal{\mu}_{n}}

\global\long\def\lmeas{\mathcal{\mu}}

\global\long\def\minpr{\mathfrak{m}}

\global\long\def\maxpr{\mathfrak{M}}

\global\long\def\boundpr{\mathfrak{U}}

\global\long\def\major#1{\mathcal{#1}}

\global\long\def\grid{\mathbb{G}}

\global\long\def\ucc{\mathrm{ucc}}

\global\long\def\Hspc{\mathscr{H}}

\global\long\def\Bspc{\mathscr{B}}

\global\long\def\Aspc{\mathscr{A}}

\global\long\def\Cspc{\mathscr{C}}

\global\long\def\gauss{\phi_{\mathfrak{g}}}

\global\long\def\Gauss{\Phi_{\mathfrak{g}}}

\global\long\def\cmpct{\varphi}

\global\long\def\Fset{\mathscr{F}}

\global\long\def\Gset{\mathscr{G}}

\global\long\def\Uset{\mathscr{U}}

\global\long\def\locp{\kappa}

\global\long\def\signal{S}

\global\long\def\cfidx{p_{0}}

\global\long\def\smlfloor#1{\lfloor#1\rfloor}

\global\long\def\smlceil#1{\lceil#1\rceil}

\global\long\def\bkt{[\,]}

\global\long\def\bktpower{\beta}

\global\long\def\ctrproc{\nu_{n}}

\newcommand{\SM}{\textnormal{(SM)}}

\newcommand{\AP}{\textnormal{(AP)}}

\newcommand{\LM}{\textnormal{(LM)}}

\newcommand{\IN}{\textnormal{(IN)}}

\newcommand{\CR}{\textnormal{(CR)}}

\global\long\def\diam{\operatorname{diam}}

\global\long\def\nseq{e}

\global\long\def\etamom{q_{0}}

\global\long\def\bandmax{r_{0}}

\global\long\def\isbigop{\lesssim_{p}}

\global\long\def\derivbnd#1{\abv m_{#1}}

\global\long\def\truncnorm#1{\smlnorm{#1}_{n}}

\global\long\def\intfnal{\iota}
\selectlanguage{british}%


\defcitealias{VVW96}{VW}

\defcitealias{Bingham87}{BGT}

\newcommand{\XREFgeneralIP}{Theorem~3.1}

\newcommand{\XREFfclt}{Proposition~2.1}

\newcommand{\XREFmaxSn}{Proposition~4.2}

\newcommand{\XREFmostsupp}{(5.4)}

\newcommand{\XREFlemmasproof}{Section~F}

\newcommand{\XREFcfprod}{9.4}

\newcommand{\XREFmgax}{7.1}

\newcommand{\XREFptwise}{Section~7}

\newcommand{\XREFfundamental}{Lemma~7.3}

\newcommand{\XREFmgbnd}{Lemma~7.4}

\newcommand{\XREFlipobj}{(C.1)}

\newcommand{\XREFliptail}{Lemma~B.1}

\newcommand{\XREFgridapprox}{Lemma~6.1}

\newcommand{\XREFalphzeroen}{Lemma~9.1(ii)}

\newcommand{\XREFmaxSnb}{Remark~4.1}

\newcommand{\XREFfourinv}{(9.1)}

\newcommand{\XREFcfprodbnd}{Lemma~F.2}

\newcommand{\XREFstmom}{Lemma~9.3(i)}

\title{Uniform Convergence Rates over Maximal Domains in Structural Nonparametric
Cointegrating Regression}

\author{James A.\ Duffy\thanks{Institute for New Economic Thinking, Oxford Martin School; and Economics
Department, University of Oxford; email: \texttt{james.duffy@economics.ox.ac.uk}.
This paper substantially revises and extends some results given in
an earlier paper of the author's \citep{Duffy13uniform}. The author
thanks Xiaohong Chen, Bent Nielsen and Peter Phillips for helpful
comments on this paper, and the earlier work. The manuscript was prepared
with \LyX{}~2.1.3 and JabRef~2.7b.}}
\maketitle
\begin{abstract}
\noindent This paper presents uniform convergence rates for kernel
regression estimators, in the setting of a structural nonlinear cointegrating
regression model. We generalise the existing literature in three ways.
First, the domain to which these rates apply is much wider than has
been previously considered, and can be chosen so as to contain as
large a fraction of the sample as desired in the limit. Second, our
results allow the regression disturbance to be serially correlated,
and cross-correlated with the regressor; previous work on this problem
(of obtaining uniform rates) having been confined entirely to the
setting of an exogenous regressor. Third, we permit the bandwidth
to be data-dependent, requiring only that it satisfy certain weak
asymptotic shrinkage conditions. Our assumptions on the regressor
process are consistent with a very broad range of departures from
the standard unit root autoregressive model, allowing the regressor
to be fractionally integrated, and to have an infinite variance (and
even infinite lower-order moments).
\end{abstract}

\section{Introduction\label{sec:intro}}

Whereas data on a stationary regressor will lie, with high probability,
within a fixed bounded interval of sufficient width, the randomly
wandering nature of an integrated process prevents it from being contained
within any such interval, no matter how wide. Consequently, the global
nonparametric estimation of a regression function taking the latter
as an argument is considerably more difficult, as it requires one
to approximate the regression function on an ever-expanding domain
-- widening probabilistically at rate $n^{1/2}$, in the unit root
case. The inherent randomness of the limiting occupation density associated
with the (standardised) regressor process poses a further challenge,
complicating the identification of domains on which observations may
be guaranteed to accumulate, in a manner that seems not to have any
parallel in the case of a stationary regressor.

This paper considers kernel nonparametric estimators of $m_{0}$ in
the nonlinear cointegrating model
\begin{equation}
y_{t}=m_{0}(x_{t})+u_{t}\label{eq:basicmodel}
\end{equation}
where $x_{t}=\sum_{s=1}^{t}v_{t}$ is the partial sum of a linear
process $\{v_{t}\}$, and $\{u_{t}\}$ is an unobserved disturbance
process. Regarding the pointwise consistency and asymptotic normality
of these estimators, we refer in particular to \citet{KMT07AS} and
\citet{WP09ET,WP09Ecta}. Our assumptions on the mechanism generating
$\{x_{t}\}$ are very general, not only permitting fractional integration
of order $d\in(-\tfrac{1}{2},\tfrac{3}{2})$ -- where $d=1$ corresponds
to the familiar unit root autoregressive model -- but also including
the case where the variance (and even lower-order moments) of $\{v_{t}\}$
do not exist.

We obtain rates of uniform convergence for our estimators on a sample-dependent
sequence of domains, which correspond as nearly as possible to the
entire empirical support of $\{x_{t}\}_{t=1}^{n}$ in the sense that
they may be chosen so as to contain as large a fraction of the data
as desired in the limit. These domains are thus maximally wide; in
contrast, previous work on uniform convergence rates in this setting
has been limited to the consideration of smaller, deterministically
expanding intervals, which necessarily contain an asymptotically negligible
fraction of the data (see \citealp{WW13ET}; \citealp{CW13mimeo};
and \citealp{GLT13mimeo}). Being able to estimate $m_{0}$ uniformly
on such wide domains should be especially useful in the context of
certain semiparametric estimation problems, such as arise when $m_{0}(x_{t})$
in \eqref{basicmodel} is replaced by the more general formulation
$m_{0}(x_{t}^{\prime}\beta_{0})$, where $x_{t}$ is a vector nonstationary
process, and both $\beta_{0}$ and $m_{0}$ are to be jointly estimated.
Clearly, only observations lying in those domains on which $m_{0}$
may be (uniformly) consistently estimated would be of any use in estimating
$\beta_{0}$; our results suggest, reassuringly, that `almost all'
the observed sample should be available for this purpose.

We further generalise previous work by permitting the regressor to
be endogenous, and the bandwidth sequence to be data-dependent in
a very general way. Endogeneity arises naturally in the setting of
cointegrating models such as \eqref{basicmodel}, which are so dynamically
under-specified as to be plausible only as a model of the long-run
equilibrium relationship between $\{y_{t}\}$ and $\{x_{t}\}$. The
exogeneity assumption typically -- though not universally, see \citet{WP09Ecta,WP13mimeo}
-- imposed in these models is thus unlikely to be satisfied in applications.
Moreover, in any application, the bandwidth used to compute a kernel
regression estimator will not be determined, \emph{a priori}, as some
function of the sample size, but will instead be chosen with at least
some reference to the sample at hand. To better accommodate this aspect
of actual empirical work, we allow the bandwidth to be functionally
dependent on the sample $\{(y_{t},x_{t})\}_{t=1}^{n}$, requiring
only that it satisfy a weak asymptotic shrinkage condition.

The proofs of these convergence rates are facilitated by a number
of new technical results. The first concerns the weak convergence
of the standardised \emph{signal} process
\begin{alignat}{1}
\loctime_{n}(a) & \defeq\frac{1}{\nseq_{n}h_{n}}\sum_{t=1}^{n}K\left(\frac{x_{t}-d_{n}a}{h_{n}}\right)\wkc\loctime(a)\label{eq:uniflaw}
\end{alignat}
in $\ell_{\ucc}(\reals)$, where: $\loctime$ denotes the occupation
density associated to the finite-dimensional limit of $X_{n}(r)\defeq d_{n}^{-1}x_{\smlfloor{nr}}$;
$\{\nseq_{n}\}$ is a norming sequence; $K\in L^{1}(\reals)$ is a
mean-zero kernel density function; and $h_{n}=o_{p}(1)$ is a smoothing
(bandwidth) sequence. This type of result is proved in \citet{Duffy14uniform}
and is reproduced as \propref{generalIP} below. \eqref{uniflaw}
may be loosely regarded as the nonstationary process counterpart of
the uniform convergence of the signal to the corresponding invariant
density that obtains when $\{x_{t}\}$ is stationary. In the present
setting, $\locest$ arises as the denominator of the Nadaraya-Watson
estimator, and so, when combined with a suitable characterisation
of the (random) support of $\loctime$, \eqref{uniflaw} allows us
to identify a sequence of domains on which the signal must accumulate
at a certain probabilistic rate.

The second class of relevant technical results supplies uniform order
estimates for
\begin{align*}
 & \frac{1}{(\nseq_{n}h_{n})^{1/2}}\sum_{t=1}^{n}f\left(\frac{x_{t}-d_{n}a}{h_{n}}\right)u_{t} &  & \frac{1}{(\nseq_{n}h_{n})^{1/2}}\sum_{t=1}^{n}g\left(\frac{x_{t}-d_{n}a}{h_{n}}\right)
\end{align*}
where $f,g\in L^{1}(\reals)$, $\int g=0$, and $\{u_{t}\}$ is weakly
dependent. These are referred to as the \emph{covariance} and \emph{zero
energy} processes, and are respectively relevant for a determination
of the uniform order of the variance and bias of a kernel regression
estimator. The estimates obtained here (\thmref{orderest} below)
appear to be new to the literature. While \citet{CW13mimeo} provide
an estimate for the covariance process when $\{x_{t}\}$ is exogenous,
our estimate holds even when $\{u_{t}\}$ is correlated with $\{x_{t}\}$,
and is within a $\log^{1/2}n$ factor of theirs. In consequence, endogeneity
of the regressor seems to penalise the rate of convergence of a kernel
regression estimator by merely a factor of $\log^{1/2}n$.

\paragraph*{Notation}

For a complete index of the notation used in this paper, see \appref{notation}
of the Supplement. For deterministic sequences $\{a_{n}\}$ and $\{b_{n}\}$,
we write $a_{n}\sim b_{n}$ if $\lim_{n\goesto\infty}a_{n}/b_{n}=1$,
and $a_{n}\asymp b_{n}$ if $\lim_{n\goesto\infty}a_{n}/b_{n}\in(-\infty,\infty)\backslash\{0\}$;
for random sequences, $a_{n}\isbigop b_{n}$ denotes $a_{n}=O_{p}(b_{n})$.
$X_{n}\wkc X$ denotes weak convergence in the sense of \citet{VVW96},
and $X_{n}\fdd X$ the convergence of finite-dimensional distributions.
For a metric space $(Q,d)$, $\ell_{\infty}(Q)$ denotes the space
of uniformly bounded functions on $Q$, equipped with the topology
of uniform convergence; while $\ell_{\ucc}(Q)$ denotes the space
of functions that are uniformly bounded on compact subsets of $Q$,
and is equipped with the topology of uniform convergence on compacta.
For $p\geq1$, $X$ a random variable, and $f:\reals\setmap\reals$,
$\smlnorm X_{p}\defeq(\expect\smlabs X^{p})^{1/p}$ and $\smlnorm f_{p}\defeq(\int_{\reals}\smlabs f^{p})^{1/p}$.
$\BI$ denotes the space of bounded and Lebesgue integrable functions
on $\reals$. $\smlfloor{\cdot}$ and $\smlceil{\cdot}$ respectively
denote the floor and ceiling functions. $C$ denotes a generic constant
that may take different values even at different places in the same
proof; $a\lesssim b$ denotes $a\leq Cb$.

\section{Discussion of results}

\subsection{Model and assumptions\label{sub:model}}

We are concerned with the estimation of $m_{0}$ in the model,
\begin{equation}
y_{t}=m_{0}(x_{t})+u_{t}\label{eq:basicreg}
\end{equation}
where $(x_{t},u_{t})$ satisfies
\begin{assumption}
\label{ass:reg}~
\begin{enumerate}
\item \label{enu:iidseq}$\{(\epsilon_{t},\eta_{t})\}$ is a bivariate i.i.d.\ sequence.
$\epsilon_{0}$ lies in the domain of attraction of a strictly stable
distribution with index $\alpha\in(0,2]$, and has characteristic
function $\psi(\lambda)\defeq\expect\e^{\i\lambda\epsilon_{0}}$ satisfying
$\psi\in L^{\cfidx}$ for some $\cfidx\geq1$. $\expect\eta_{0}=0$,
$\expect\smlabs{\epsilon_{0}\eta_{0}}<\infty$, and $\expect\smlabs{\eta_{0}}^{\etamom}<\infty$
for some $\etamom>2$.
\item \label{enu:regproc}$\{x_{t}\}$ is generated according to 
\begin{align}
x_{t} & \defeq\sum_{s=1}^{t}v_{s} & v_{t} & \defeq\sum_{k=0}^{\infty}\phi_{k}\epsilon_{t-k},\label{eq:regproc}
\end{align}
and either

\begin{enumerate}
\item \label{enu:SM}$\alpha\in(1,2]$, $\sum_{k=0}^{\infty}\smlabs{\phi_{k}}<\infty$
and $\phi\defeq\sum_{k=0}^{\infty}\phi_{k}\neq0$; or
\end{enumerate}

$\phi_{k}\sim k^{H-1-1/\alpha}\pi_{k}$ for some $\{\pi_{k}\}_{k\geq0}$
strictly positive and slowly varying at infinity, with
\begin{enumerate}[resume]
\item \label{enu:LM}$H>1/\alpha$; or
\item \label{enu:AP}$H<1/\alpha$ and $\sum_{k=0}^{\infty}\phi_{k}=0$.
\end{enumerate}

In both cases \enuref{LM} and \enuref{AP}, $H\in(\tfrac{1}{3},1)$.

\item \label{enu:dist}$\{u_{t}\}$, the regression disturbance, is the
linear process
\begin{equation}
u_{t}\defeq\sum_{k=0}^{\infty}\theta_{k}\eta_{t-k}\label{eq:disturbance}
\end{equation}
with $\sum_{k=0}^{\infty}\smlabs{\theta_{k}}k^{7/6}<\infty$.
\end{enumerate}
\end{assumption}
 
\begin{rem}
The preceding conditions may be compared with those imposed by \citet{WP09Ecta,WP13mimeo};
but quite unlike those authors, we do \emph{not} require $\expect\epsilon_{0}^{2}<\infty$.
Although our assumptions are consistent with substantial departures
from the standard unit root model -- which here coincides with \enuref{regproc}\enuref{SM}
with $\alpha=2$ -- $\{x_{t}\}$ is in all cases a partial sum process,
and this feature of the generating mechanism identifies \eqref{basicreg}
as a \emph{nonlinear cointegrating regression}, in the terminology
of \citet{PP01Ecta}. To allow the alternative forms of \enuref{regproc}
to be more concisely referenced, we shall regard part~\enuref{SM}
as corresponding to the case where $H=1/\alpha$.

The arguments used in this paper could be adapted to derive our main
results when $H\in(0,\frac{1}{3}]$. However, for $H$ falling within
this range, certain simplifications -- resulting, in particular, from
parts~\enuref{dm2} and \enuref{dna} of \lemref{summable} below
-- are unavailable, and the statement of our results would take a
more complicated form. To keep this paper to a reasonable length,
we have therefore restricted to $H\in(\tfrac{1}{3},1)$. (This restriction
is also necessary for certain related results in the literature: in
particular, see Theorems~4 and 5 in \citealp{Jeg08CDFP}.)
\end{rem}
 
\begin{rem}
Part~\enuref{dist} permits the regression disturbance to be serially
dependent, and cross-correlated with the regressor; \eqref{basicreg}
is thus a \emph{structural} model. \citet{WP13mimeo} allow $\{u_{t}\}$
to be generated in a slightly more general manner, according to 
\[
u_{t}=\sum_{k=0}^{\infty}[\theta_{k}\eta_{t-k}+\vartheta_{k}g(\epsilon_{t-k})]\eqdef u_{1t}+u_{2t}.
\]
By considering separately the cases in which $u_{t}=u_{1t}$ and $u_{t}=u_{2t}$,
it is easily seen that Theorems~\ref{thm:orderest} and \ref{thm:rates}
below hold, without any modification, when $u_{t}=u_{1t}+u_{2t}$,
provided that $\sum_{k=0}^{\infty}\smlabs{\vartheta_{k}}k^{7/6}<\infty$
and $g$ is such that $\expect g^{2}(\epsilon_{0})<\infty$.

Instead of \enuref{dist}, we might have required $\{u_{t},\tilde{\filt}_{t+1}\}$
to be a martingale difference sequence (m.d.s.), where $\tilde{\filt}_{t}\defeq\sigma(\{x_{s},u_{s-1}\}_{s\leq t})$
(see e.g.\ \citealp{PP01Ecta}, Ass.\ 2.1); we say that $\{u_{t}\}$
is an \emph{exogenous m.d.s.}\ in this case. While this alternative
assumption would be very convenient, it seems rather restrictive in
the setting of such an `under-specified' model as \eqref{basicreg},
in which any short-run dynamics affecting the relationship between
$y_{t}$ and $x_{t}$ must be absorbed into $u_{t}$.
\end{rem}
 
\begin{rem}
The requirement that $\sum_{k=0}^{\infty}\smlabs{\theta_{k}}k^{7/6}<\infty$
is stronger than is necessary to ensure the pointwise consistency
and asymptotic normality of kernel regression estimators in this setting:
see for example \citet{WP13mimeo}, who merely assume that $\sum_{k=0}^{\infty}\smlabs{\theta_{k}}k^{1/4}<\infty$.
Were \enuref{dist} to be relaxed in this direction, then the arguments
used to prove the main results of this paper could still be applied,
but the rates of convergence obtained would be complicated by the
presence of an additional term, the magnitude of which would depend,
in a somewhat complicated manner, on the rate at which $\theta_{k}\goesto0$
as $k\goesto\infty$.
\end{rem}

We shall consider both local level (Nadaraya-Watson) and local linear
estimators of $m_{0}$, to be denoted by $\hat{m}$ and $\hat{m}_{L}$
respectively. To facilitate nonparametric estimation, we require
\begin{assumption}
\label{ass:smooth}$m_{0}$ is twice continuously differentiable.
\end{assumption}
 
\begin{rem}
The preceding is stronger than is necessary for a determination of
the convergence rates of these estimators; our arguments would also
permit a derivation of these rates when \assref{smooth} is relaxed
to the Hlder continuity of $m_{0}$ or of its derivatives. We have
refrained from doing so only in order to permit our convergence rates
to be concisely stated.
\end{rem}

The construction of both estimators involves the use of a smoothing
kernel $K$, and a bandwidth sequence $\{h_{n}\}$. In order to state
our assumptions on $K$, let $\BIl$ denote the set of bounded, integrable
and Lipschitz continuous functions on $\reals$, and recall that $\{\nseq_{n}\}$
is the norming sequence that appears in \eqref{uniflaw} above (see
also \eqref{dkek} below).
\begin{assumption}
\label{ass:estim}~
\begin{enumerate}
\item $K\in\BIl$ is compactly supported, with $\int K=1$ and $\int xK(x)\diff x=0$.
\item $h_{n}\in\Hspc_{n}\defeq[\blw h_{n},\abv h]$ with probability approaching
1 (w.p.a.1), where $\abv h<\infty$, and $\blw h_{n}^{-1}\lesssim\nseq_{n}n^{-2\bandmax}$
for some $\bandmax>0$.
\end{enumerate}
\end{assumption}
 
\begin{rem}
The Lipschitz continuity and compact support of $K$ ease some of
our arguments, but are certainly not necessary for the most fundamental
results which underpin our derivations. For example, \XREFgeneralIP{}
in \citet{Duffy14uniform} requires only that $K$ have one-sided
Lipschitz approximants, a rather weak condition that \emph{is} consistent
with the presence of simple discontinuities.
\end{rem}
 
\begin{rem}
An important feature of the present work, relative to the preceding
literature on nonlinear cointegration, is that we permit the bandwidth
sequence to be random, and thus data-dependent, requiring only that
it take values lying in the (growing) interval $[\blw h_{n},\abv h_{n}]$,
w.p.a.1. This is of considerable utility in applications, where $h_{n}$
will typically be chosen with at least some reference to the sample
at hand, making the assumption that $\{h_{n}\}$ is a `given' deterministic
sequence quite unrealistic. In the i.i.d.\ regressor case, results
of this kind are given in \citet{EM05AS}. Note that restricting to
deterministic bandwidth sequences would \emph{not} help us to obtain
better rates of convergence than are given in \thmref{rates} below.
\end{rem}

Assumptions~\ref{ass:reg}--\ref{ass:estim} are maintained throughout
the paper, even when no explicit reference is made to them. We shall
treat the parameters (including $H$ and $\alpha$) describing the
data generating mechanism as `fixed', ignoring the dependence of
any constants on these.

\subsection{Asymptotic behaviour of the regressor process\label{sub:asympreg}}

Before proceeding to an account of our main results, we describe the
limiting behaviour of the standardised regressor process $X_{n}(r)\defeq d_{n}^{-1}x_{\smlfloor{nr}}$
that is entailed by our assumptions, and which is fundamental to our
results. (The required norming sequence $\{d_{n}\}$ is given in \eqref{dkek}
below.)

Part~\enuref{iidseq} of \assref{reg} implies that there exists
a slowly varying sequence $\{\varrho_{k}\}$ such that
\begin{equation}
\frac{1}{n^{1/\alpha}\varrho_{n}}\sum_{t=1}^{\smlfloor{nr}}\epsilon_{t}\fdd Z_{\alpha}(r)\label{eq:stablecvg}
\end{equation}
where $Z_{\alpha}$ denotes an \emph{$\alpha$-stable Lvy motion}
on $\reals$, with $Z_{\alpha}(0)=0$. That is, the increments of
$Z_{\alpha}$ are stationary, and for any $r_{1}<r_{2}$ the characteristic
function of $Z_{\alpha}(r_{2})-Z_{\alpha}(r_{1})$ has the logarithm
\begin{equation}
-(r_{2}-r_{1})c\smlabs{\lambda}^{\alpha}\left[1-\i\beta\tan\left(\frac{\pi\alpha}{2}\right)\right]\label{eq:incrcf}
\end{equation}
where $\beta\in[-1,1]$ and $c>0$; following \citet[p.~1773]{Jeg04},
we impose the further restriction that $\beta=0$ when $\alpha=1$.
We shall also require that $\{\varrho_{k}\}$ be chosen such that
$c=1$ here, which provides a convenient normalisation for the scale
of $Z_{\alpha}$. (Thus when $\alpha=2$, $Z_{\alpha}$ corresponds
to a Brownian motion with variance $2$.) Let $X$ denote the \emph{linear
fractional stable motion }(LFSM)
\begin{align}
X(r) & \defeq\int_{0}^{r}(r-s)^{H-1/\alpha}\diff Z_{\alpha}(s)\label{eq:Xdef}\\
 & \qquad\qquad+\int_{-\infty}^{0}[(r-s)^{H-1/\alpha}-(-s)^{H-1/\alpha}]\diff Z_{\alpha}(s)\nonumber 
\end{align}
with the convention that $X=Z_{\alpha}$ when $H=1/\alpha$. (See
\citealp{ST94}, for a detailed discussion of the LFSM; note that
when $\alpha=2$, $X$ is a fractional Brownian motion.) Associated
to $X$ is the occupation density (\emph{local time}) process $\loctime\defeq\{\loctime(a)\}_{a\in\reals}$,
a process which, almost surely, has continuous paths and satisfies
\begin{equation}
\int_{\reals}f(x)\loctime(x)\diff x=\int_{0}^{1}f(X(r))\diff r,\qquad\forall f\text{ bounded, measurable}.\label{eq:otf}
\end{equation}
(See Theorem~0 in \citealp{Jeg04}.)

Now let $\{c_{k}\}$ denote a sequence with $c_{0}=1$ and \setjot{2}
\begin{equation}
c_{k}=\begin{cases}
\phi & \text{if }H=1/\alpha\\
\smlabs{H-1/\alpha}^{-1}k^{H-1/\alpha}\pi_{k} & \text{otherwise.}
\end{cases}\label{eq:ck}
\end{equation}
By Karamata's theorem \citep[Thm.~1.5.11]{Bingham87}, $\sum_{l=0}^{k}\phi_{k}\sim c_{k}$
as $k\goesto\infty$.\restorejot{} Set 
\begin{align}
d_{k} & \defeq k^{1/\alpha}c_{k}\varrho_{k} & \nseq_{k} & \defeq kd_{k}^{-1},\label{eq:dkek}
\end{align}
and note that the sequences $\{c_{k}\}$, $\{d_{k}\}$ and $\{\nseq_{k}\}$
are regularly varying with indices $H-1/\alpha$, $H$ and $1-H$
respectively. The following is a special case of \XREFfclt{} and
\XREFgeneralIP{} in \citet{Duffy14uniform}.
\begin{prop}
\label{prop:generalIP} For every $f\in\BIl$, 
\begin{gather}
\loctime_{n}^{f}(a)\defeq\frac{1}{\nseq_{n}h_{n}}\sum_{t=1}^{n}f\left(\frac{x_{t}-d_{n}a}{h_{n}}\right)\wkc\loctime(a)\int_{\reals}f\label{eq:uniformlaw}
\end{gather}
on $\ell_{\ucc}(\reals)$, jointly with $X_{n}(r)\defeq d_{n}^{-1}x_{\smlfloor{nr}}\fdd X(r)$.
\end{prop}

\subsection{Order estimates\label{sub:order}}

Our rates of convergence will be obtained with the aid of the following
order estimates, for the covariance and zero energy processes respectively,
which appear to be new in the literature.
\begin{thm}
\label{thm:orderest}Suppose $f\in\BIl$. Then
\begin{gather}
\frac{1}{(\nseq_{n}h_{n})^{1/2}}\sup_{a\in\reals}\abs{\sum_{t=1}^{n}f\left(\frac{x_{t}-d_{n}a}{h_{n}}\right)u_{t}}\isbigop(1+n^{1/\etamom-\bandmax})\log n;\label{eq:covorderest}
\end{gather}
and if additionally $\int_{\reals}f=0$ and $\int\smlabs{f(x)x}\diff x<\infty$,
\begin{equation}
\frac{1}{(\nseq_{n}h_{n})^{1/2}}\sup_{a\in\reals}\abs{\sum_{t=1}^{n}f\left(\frac{x_{t}-d_{n}a}{h_{n}}\right)}\isbigop\log n.\label{eq:0enorderest}
\end{equation}

\end{thm}

See \secref{orderest} for the proof. In both cases the assumed smoothness
of $f$ permits the supremum over $\reals$ to be effectively reduced
to a maximum over a sequence of finite sets, $\{\Fset_{n}\}$. For
the zero energy process, the requisite bound over $\Fset_{n}$ is
provided (essentially) by \XREFmaxSn{} in \citet{Duffy14uniform}.
It turns out that a counterpart of this result is available for the
covariance process, but its application requires that a truncation
be first applied to $\{u_{t}\}$. In order to state the result, let
$\eta_{t}^{(\leq)}$ denote an appropriately truncated (and centred)
version of $\eta_{t}$ (see \eqref{etatrunc} below), such that $\expect\eta_{t}^{(\leq)}=0$
and $\truncnorm{\eta}\defeq\smlnorm{\eta_{0}^{(\leq)}}_{\infty}<\infty$.
For $u_{t}^{(\leq)}\defeq\sum_{k=0}^{n}\theta_{k}\eta_{t-k}^{(\leq)}$,
define
\begin{equation}
\S_{n}f\defeq\sum_{t=1}^{n}f(x_{t})u_{t}^{(\leq)},\label{eq:Snf}
\end{equation}
and for $\Fset\subset\BI$, set
\begin{equation}
\delta_{n}(\Fset)\defeq\truncnorm{\eta}\smlnorm{\Fset}_{\infty}+[\truncnorm{\eta}+\nseq_{n}^{1/2}](\smlnorm{\Fset}_{1}+\smlnorm{\Fset}_{2})\label{eq:delnf}
\end{equation}
where $\smlnorm{\Fset}\defeq\sup_{f\in\Fset}\smlnorm f$.
\begin{prop}
\label{prop:covtruncord}Suppose $\Fset_{n}\subset\BI$ with $\#\Fset_{n}\lesssim n^{C}$.
Then
\[
\max_{f\in\Fset_{n}}\smlabs{\S_{n}f}\isbigop\delta_{n}(\Fset_{n})\log n.
\]

\end{prop}

The proof of this result appears in \secref{covproof}.
\begin{rem}
\label{rem:exogcase}If $\{u_{t}\}$ is an exogenous m.d.s., then
$\{f(x_{t})u_{t}\}$ itself forms a m.d.s., and an application of
Freedman's \citeyearpar[Thm.~1.6]{Freedman75AP} inequality permits
the $\log n$ factor on the right side of \eqref{covorderest} to
be reduced to $\log^{1/2}n$; see \citet[Thm.~2.1]{WC14Bern}. In
the present case, however, $\{f(x_{t})u_{t}\}$ is not a m.d.s., precluding
a direct application of such a result. Instead, we shall rely on the
combination of a suitable subgaussian tail inequality for martingales
\citep[Thm.~2.1]{BT08AAP} -- which implies \lemref{mgmax} below
-- and a martingale decomposition of $\sum_{t=1}^{n}f(x_{t})u_{t}$.
\end{rem}

\subsection{Rates of uniform convergence\label{sub:rates}}

The essential features of the problem become apparent when we consider
the Nadaraya-Watson estimator, which admits the decomposition
\begin{align}
\hat{m}(x)-m_{0}(x) & =\frac{\sum_{t=1}^{n}K_{h_{n}}(x_{t}-x)[m_{0}(x_{t})-m_{0}(x)]}{\sum_{t=1}^{n}K_{h_{n}}(x_{t}-x)}+\frac{\sum_{t=1}^{n}K_{h_{n}}(x_{t}-x)u_{t}}{\sum_{t=1}^{n}K_{h_{n}}(x_{t}-x)}\nonumber \\
 & \eqdef\frac{\objlabel{\Psi_{1n}}{bias}(x)}{\objlabel{\Psi_{3n}}{denom}(x)}+\frac{\objlabel{\Psi_{2n}}{variance}(x)}{\objref{denom}(x)},\label{eq:decmp}
\end{align}
where $K_{h}(y)\defeq h^{-1}K(h^{-1}y)$. We shall now examine each
of $\objref{bias}$, $\objref{variance}$ and $\objref{denom}$, in
turn: compared with stationary regressor case, the treatment of the
denominator poses some unique challenges here, and so we turn to it
first.

\paragraph*{Denominator}

Set $\locest(a)\defeq\nseq_{n}^{-1}\sum_{t=1}^{n}K_{h_{n}}(x_{t}-d_{n}a)$,
and define
\begin{equation}
A_{n}^{\varepsilon}\defeq\{x\in\reals\mid\locest(d_{n}^{-1}x)\geq\varepsilon\}.\label{eq:An}
\end{equation}
Noting the different standardisations of $\locest$ and $\objref{denom}$,
we see that
\begin{equation}
\sup_{x\in A_{n}^{\varepsilon}}\objref{denom}^{-1}(x)=\Bigl(\inf_{x\in A_{n}^{\varepsilon}}\objref{denom}(x)\Bigr)^{-1}\leq\varepsilon^{-1}e_{n}^{-1}.\label{eq:denomord}
\end{equation}
Thus $A_{n}^{\varepsilon}$ describes a subset of $\reals$ -- which
depends on the trajectory of $\{x_{t}\}_{t=1}^{n}$ -- on which the
order of $\objref{denom}^{-1}(x)$ may be uniformly controlled. Importantly,
$\varepsilon>0$ may be chosen (sufficiently small) such that $A_{n}^{\varepsilon}$
contains as large a fraction of the sample as desired, in the limit
as $n\goesto\infty$, in the sense that
\begin{equation}
\limsup_{n\goesto\infty}\Prob\left\{ \frac{1}{n}\sum_{t=1}^{n}\indic\{x_{t}\notin A_{n}^{\varepsilon}\}\geq\delta\right\} \leq\delta\label{eq:mostsupp}
\end{equation}
for any given $\delta>0$: see the arguments used to verify \XREFmostsupp{}
in \citet{Duffy14uniform}.

In general, $A_{n}^{\varepsilon}$ will be a union of disjoint (closed)
intervals, even for large $n$. This is necessarily the case when
$H<1/\alpha$, since in this case $X$ has discontinuous sample paths
(\citealp{ST94}, Example~10.2.5), and so the support of $\loctime$,
which is contained in the range of $X$, will typically contain gaps.
However, in the special case where $\expect\epsilon_{0}^{2}<\infty$
(implying $\alpha=2$) and $H=\tfrac{1}{2}$, it is possible to replace
$A_{n}^{\varepsilon}$ by a sequence of (connected) intervals. In
order to state a result to this effect, let $R_{n}^{\varepsilon}\defeq[x_{(1)},x_{(n)}]_{\varepsilon}$,
where $[a,b]_{\varepsilon}\defeq[(1-\varepsilon)a,(1-\varepsilon)b]$
and let $\{x_{(i)}\}_{i=1}^{n}$ denote the order statistics of the
sample $\{x_{t}\}_{t=1}^{n}$.
\begin{prop}
\label{prop:Raydenom}Suppose $H=\tfrac{1}{2}$ and $\expect\epsilon_{0}^{2}<\infty$,
and let $\delta>0$ be given. Then for every $\varepsilon>0$, \eqref{denomord}
holds with $R_{n}^{\varepsilon}$ in place of $A_{n}^{\varepsilon}$;
and for $\varepsilon>0$ sufficiently small, \eqref{mostsupp} holds
with $R_{n}^{\varepsilon}$ in place of $A_{n}^{\varepsilon}$.
\end{prop}
 
\begin{rem}
This result is a essentially a consequence \citeauthor{Ray1963IJM}'s
\citeyearpar{Ray1963IJM} theorem, which implies that the local time
$\loctime$ of a diffusion $J$ is strictly positive on the interior
of the range of $J$. In the setting of the present paper, this seems
to be applicable only in the case where $X$ is a Brownian motion.
However, it also applies in the important case where
\[
X(r)=\int_{0}^{r}\e^{\locp(r-s)}\diff B(s)
\]
for $\locp\in\reals$ and $B$ a Brownian motion. Such a process arises
as the weak limit of $X_{n}$, under the hypotheses of \propref{Raydenom},
if $\{x_{t}\}$ is generated according to $x_{t}=x_{t-1}+\rho_{n}v_{t}$,
where $\rho_{n}=1+\frac{\locp}{n}$: see \citet{WP09Ecta}. (To extend
our results to this case would require very little modification to
our arguments; indeed, we explicitly considered such a data generating
mechanism in an earlier version of this paper.) Whether such a characterisation
of the support of $\loctime$ is available in other cases where $X$
has continuous sample paths, most notably when $\alpha=2$ and $H\neq1/2$
-- i.e.\ when $X$ is a fractional Brownian motion -- seems to be
an open question.
\end{rem}
 
\begin{rem}
\label{rem:detintvl}In view of \eqref{mostsupp}, the volumes of
both $A_{n}^{\varepsilon}$ and $R_{n}^{\varepsilon}$ must expand,
probabilistically, at rate $d_{n}$ as $n\goesto\infty$. However,
even under the hypotheses of \propref{Raydenom}, $A_{n}^{\varepsilon}$
could \emph{not} be replaced by a sequence of \emph{deterministic}
intervals $[-a_{n},a_{n}]$ whose endpoints diverge at rate $d_{n}$.
Indeed, suppose that $a_{n}=C_{0}n^{1/2}$ for some $C_{0}>0$. Then
by $X_{n}\wkc X$ and the reflection principle \citep[Prop.~III.3.7]{RY99book},
\begin{align}
\Prob\left\{ \max_{1\leq t\leq n}x_{t}\leq\frac{C_{0}n^{1/2}}{2}\right\} \goesto\Prob\left\{ \sup_{r\in[0,1]}X(r)\leq\frac{C_{0}}{2}\right\}  & =1-2\Prob\left\{ X(1)>\frac{C_{0}}{2}\right\} =2\Gauss\left(\frac{C_{0}}{2}\right)-1>0\label{eq:boundedocc}
\end{align}
for every $C_{0}>0$, no matter how small; here $\Gauss$ denotes
the standard normal c.d.f. With nonzero probability, $\{x_{t}\}_{t=1}^{n}$
\emph{never} visits $[\frac{1}{2}C_{0}n^{1/2},C_{0}n^{1/2}],$ and
so the signal is forever negligible within this range. This accounts
for why earlier work on this problem (e.g.\ \citealp{WW13ET}; \citealp{CW13mimeo};
and \citealp{GLT13mimeo}), which considered deterministic intervals
of this form, has been restricted to domains whose volume grows at
a rate strictly slower than $d_{n}$, which necessarily contain a
vanishingly small fraction of the observed $\{x_{t}\}_{t=1}^{n}$
as $n\goesto\infty$.
\end{rem}

\paragraph*{Numerator}

To provide a measure of the `regularity' of $m_{0}$ over a given
domain, we associate to $m_{0}$ the mappings $\derivbnd 1,\derivbnd 2:\pset(\reals)\setmap\reals_{+}\union\{\infty\}$,
defined by
\begin{align}
\derivbnd 1(A) & \defeq\sup_{x\in A}\smlabs{m_{0}^{\prime}(x)} & \derivbnd 2(A) & \defeq\sup_{x\in A}\smlabs{m_{0}^{\prime\prime}(x)}.\label{eq:mregularity}
\end{align}
Let $\tilde{A}_{n}^{\varepsilon}\defeq\{x\in\reals\mid d(x,A_{n}^{\varepsilon})\leq c_{K}h\}$,
where $c_{K}$ is chosen such that the support of $K$ is contained
in $[-c_{K},c_{K}]$. Then a Taylor series expansion of $m_{0}$ around
$x_{t}$, for each $t\in\{1,\ldots,n\}$, yields the estimate
\begin{align}
\sup_{x\in A_{n}^{\varepsilon}}\smlabs{\objref{bias}} & \leq h_{n}\derivbnd 1(\tilde{A}_{n}^{\varepsilon})\sup_{x\in\reals}\abs{\sum_{t=1}^{n}K_{h_{n}}^{[1]}(x_{t}-x)}+h_{n}^{2}\derivbnd 2(\tilde{A}_{n}^{\varepsilon})\sup_{x\in\reals}\sum_{t=1}^{n}\smlabs{K_{h_{n}}^{[2]}(x_{t}-x)},\label{eq:psi1bnd}
\end{align}
where $f^{[p]}$ denotes $x\elmap x^{p}f(x)$. Applying \thmref{orderest}
and \propref{generalIP} to the first and second terms on the right,
respectively, then gives
\begin{equation}
\frac{1}{\nseq_{n}}\sup_{x\in A_{n}^{\varepsilon}}\smlabs{\objref{bias}}\isbigop h_{n}^{2}\left[\derivbnd 1(\tilde{A}_{n}^{\varepsilon})\frac{\log n}{h_{n}^{3/2}\nseq_{n}^{1/2}}+\derivbnd 2(\tilde{A}_{n}^{\varepsilon})\right].\label{eq:biasorder}
\end{equation}
Similarly, we have by \thmref{orderest} that 
\begin{equation}
\frac{1}{\nseq_{n}}\sup_{x\in A_{n}^{\varepsilon}}\smlabs{\objref{variance}}\isbigop(1+n^{1/\etamom-\bandmax})\frac{\log n}{(\nseq_{n}h_{n})^{1/2}}.\label{eq:varorder}
\end{equation}

In view of \eqref{denomord}, the rate at which $\hat{m}(x)$ converges
uniformly to $m_{0}(x)$ on $A_{n}^{\varepsilon}$ is given by the
sum of the right sides of \eqref{biasorder} and \eqref{varorder}.
In giving a formal statement of our results below, we assume that
$\etamom$ and $\bandmax$ are such that $n^{1/\etamom-\bandmax}\lesssim1$,
so that the right side of \eqref{varorder} takes a simplified form.
\begin{thm}
\label{thm:rates}Suppose $\bandmax\geq\etamom^{-1}$. Then for every
$\varepsilon>0$,
\begin{align}
\varepsilon\cdot\sup_{x\in A_{n}^{\varepsilon}}\smlabs{\hat{m}(x)-m_{0}(x)}\isbigop h_{n}^{2}\left[\derivbnd 1(\tilde{A}_{n}^{\varepsilon})\frac{\log n}{h_{n}^{3/2}\nseq_{n}^{1/2}}+\derivbnd 2(\tilde{A}_{n}^{\varepsilon})\right]+\frac{\log n}{(\nseq_{n}h_{n})^{1/2}}\label{eq:ndrate}
\end{align}
and
\begin{equation}
\varepsilon\cdot\sup_{x\in A_{n}^{\varepsilon}}\smlabs{\hat{m}_{L}(x)-m_{0}(x)}\isbigop h_{n}^{2}\derivbnd 2(\tilde{A}_{n}^{\varepsilon})+\frac{\log n}{(\nseq_{n}h_{n})^{1/2}}\label{eq:linrate}
\end{equation}

\end{thm}
 
\begin{rem}
These uniform convergence rates agree almost exactly with their pointwise
counterparts (see e.g.\ \citealp{WP11ET}), except for
\begin{enumerate}
\item the presence of the $\log n$ factors; and
\item the dependence of the bias terms on the \emph{suprema} of $m_{0}^{\prime}$
and $m_{0}^{\prime\prime}$ over $\tilde{A}_{n}^{\varepsilon}$.
\end{enumerate}
Thus, as the support of $\{x_{t}\}$ -- to which $A_{n}^{\varepsilon}$
is an approximation -- expands over the line, these `uniform' bias
terms may shrink less rapidly than their pointwise counterparts, depending
on the tail behaviour of the derivatives of $m_{0}$.
\end{rem}
 
\begin{rem}
\label{rem:exog} After the manuscript of this paper had been completed,
we obtained a copy of an unpublished manuscript by \citet{LCW14mimeo},
who determine the uniform rate of convergence of $\hat{m}_{L}$ when
$\{x_{t}\}$ is exogenous ($\{u_{t}\}$ is a heteroskedastic m.d.s.)\ and
$\{h_{n}\}$ is a deterministic sequence. (Rather than a sequence
of random domains such as $\{A_{n}^{\varepsilon}\}$, they consider
only a deterministic sequence of intervals, with the consequences
discussed in \remref{detintvl} above.) Due to the assumed exogeneity,
the $\log n$ factor appearing in the second terms on the right sides
of \eqref{ndrate} and \eqref{linrate} can be improved to $\log^{1/2}n$,
as per \remref{exogcase} above. The presence of endogeneity would
thus seem to penalise the rate of convergence of these estimators
by at worst a factor of $\log^{1/2}n$.

Underpinning these authors' derivations is an analogue of our \propref{generalIP}
(their Theorem~2.1), which is worked out under quite different assumptions
on the regressor process than are imposed here. In this regard, we
may note particularly their requirement that there exist a sequence
of processes $\{X_{n}^{\ast}\}$ with $X_{n}^{\ast}\eqdist X$, and
a $\delta>0$ such that
\begin{equation}
\sup_{r\in[0,1]}\smlabs{X_{n}(r)-X_{n}^{\ast}(r)}=o_{\as}(n^{-\delta}),\label{eq:LCWcondition}
\end{equation}
a condition which excludes a large portion of the processes considered
in this paper, for which merely $X_{n}\fdd X$ is available (this
is particularly true when $H<1/\alpha$ and $\alpha\in(0,2)$, since
in this case the sample paths of $X$ are unbounded: see \citealp{ST94},
Example~10.2.5). On the other hand, our results do \emph{not} subsume
those of \citet{LCW14mimeo}, since those authors do not require $\{v_{t}\}$
to be a linear process; there is thus only a partial overlap between
the class of processes considered in this paper, and in theirs.
\end{rem}
 
\begin{rem}
Provided that $\varepsilon>0$ is fixed, and $\derivbnd 1$ and $\derivbnd 2$
are bounded on $\reals$ -- which is perfectly consistent with linear
or sublinear growth in the tails of $m_{0}$ -- \thmref{rates} implies
that requiring convergence on a domain almost as large as the range
of $\{x_{t}\}$ does not penalise the convergence \emph{rate} of either
estimator, relative to the rate that could be proved on an interval
of fixed with. This might seem to contrast markedly with the situation
when $\{x_{t}\}$ is stationary, where necessarily slower rates of
convergence hold on domains that expand with the sample size, as the
estimator is pushed into regions where $\{x_{t}\}$ has a progressively
smaller density (see e.g.\ \citealp[Thm.~8]{Hansen08ET}; \citealp[Thm.~1]{Krist09ET};
and \citealp[Thm.~2.1]{LLL12ET}). However, this phenomenon would
re-emerge here if we were to let $\varepsilon=\varepsilon_{n}\goesto0$
as $n\goesto\infty$: indeed, it is immediate from \eqref{ndrate}
and \eqref{linrate} that the rates of uniform convergence of our
estimators, over the domains $\{A_{n}^{\varepsilon_{n}}\}$, would
be slowed by a factor of $\varepsilon_{n}^{-1}$ in this case.
\end{rem}
 
\begin{rem}
In a recent paper, \citet{CW13mimeo} argue that while both $\hat{m}$
and $\hat{m}_{L}$ enjoy similar pointwise bias properties, the latter
enjoys markedly better performance than the former, so far as the
uniform behaviour of their respective bias terms is concerned. This
conclusion is \emph{partly} borne out by \thmref{rates}; but the
improved order estimate obtained here for the linear bias term (the
first element on the right side of \eqref{ndrate}) indicates that
this judgement may need to somewhat qualified. In particular, if both
$\derivbnd 1(A_{n}^{\varepsilon})$ and $\derivbnd 2(\tilde{A}_{n}^{\varepsilon})$
are of a comparable magnitude, then $\nseq_{n}h_{n}^{3}\log^{-2}n\inprob\infty$
will ensure that the second order bias (the second term on the right
side of \eqref{ndrate}) dominates the linear bias term; this is scarcely
less restrictive than the condition that $\nseq_{n}h_{n}^{3}\inprob\infty$
that is required for this conclusion when only the \emph{pointwise}
performance of these estimators is in issue (see \citealp{WP11ET}). 
\end{rem}

See \secref{rateproof} for the proofs of \propref{Raydenom} and
\thmref{rates}.

\subsection{An alternative perspective on our results}

There is another way of viewing the estimation problem considered
in this paper, which highlights the connections between our results,
and those which are obtained when $\{x_{t}\}$ is stationary. Defining
a sequence of regression functions $m_{n}(x)\defeq m_{0}(d_{n}x)$,
the model \eqref{basicreg} can be rewritten as
\begin{equation}
y_{t}=m_{0}(x_{t})+u_{t}=m_{n}(d_{n}^{-1}x_{t})+u_{t}=m_{n}(x_{nt})+u_{t},\label{eq:transreg}
\end{equation}
where $x_{nt}\defeq d_{n}^{-1}x_{t}$. Taking $b_{n}=d_{n}^{-1}h_{n}$,
we see that 
\begin{equation}
\frac{1}{nb_{n}}\sum_{t=1}^{n}K\left(\frac{x_{nt}-x}{b_{n}}\right)=\frac{1}{\nseq_{n}h_{n}}\sum_{t=1}^{n}K\left(\frac{x_{t}-d_{n}x}{h_{n}}\right)\wkc\loctime(x)\label{eq:denscon}
\end{equation}
in $\ell_{\infty}(\reals)$ by \propref{generalIP}. In light of this,
we might regard the $\{x_{nt}\}$'s as being drawn from a spatial
distribution with marginal density $\loctime(x)$ -- just as stationary
regressors would be drawn from a distribution with marginal density
$p(x)$. (Restricting ourselves to $A_{n}^{\varepsilon}$ or $R_{n}^{\varepsilon}$
yields a domain on which the density $\loctime$ can be bounded away
from 0, which is equally desirable in the stationary regressor case.)

Now suppose that almost nothing is known about $\{x_{nt}\}$, beyond
the fact that a convergence result of the same kind as \eqref{denscon}
holds (together with some knowledge of the support of $\loctime$).
Is this sufficient to determine the rate at which the local linear
estimator of $m_{n}$, computed from $\{y_{t},x_{nt}\}_{t=1}^{n}$,
converges uniformly to $m_{n}$? If $\{u_{t}\}$ is an exogenous m.d.s.,
then this is indeed the case. Supposing 
\[
\sup_{x\in\reals}\smlabs{m_{n}^{\prime\prime}(x)}\le\alpha_{n}
\]
then as per \remref{exog}, we would have
\begin{equation}
\sup_{x\in A}\smlabs{\hat{m}_{L}(x)-m_{n}(x)}\isbigop b_{n}^{2}\alpha_{n}+\frac{\log^{1/2}n}{n^{1/2}b_{n}^{1/2}}=\alpha_{n}^{1/5}\frac{\log^{2/5}n}{n^{2/5}},\label{eq:exoglinrate}
\end{equation}
for any set $A$ on which $\loctime$ can be bounded away from zero;
the final equality follows if $b_{n}$ is chosen so as to balance
the order of bias and variance terms.

The rate of convergence thus depends only on how the `complexity'
of $m_{n}$ -- as measured here by $\alpha_{n}$ -- varies with $n$.
In the stationary setting, $m_{n}$ is (typically) a \emph{fixed}
function, and so $\alpha_{n}=\alpha_{0}$, a constant. In this case,
the right side of \eqref{exoglinrate} agrees precisely with \citeauthor{Stone1982AS}'s
\citeyearpar{Stone1982AS} minimax optimal rate for twice-continuously
differentiable functions. On the other hand, when $\{x_{t}\}$ is
integrated (or near-integrated), the manner in which $m_{n}$ is constructed
from a \emph{fixed} $m_{0}$ gives $\alpha_{n}=\alpha_{0}d_{n}^{2}$,
and the best obtainable rate is $O_{p}(\nseq_{n}^{-2/5}\log^{2/5}n)$,
as per \citet{CW13mimeo}. This leads us to believe that the minimax
optimality properties of local polynomial regression, for the estimation
of functions belonging to Hlder classes, should extend quite straightforwardly
to the case of an integrated regressor.

It remains to be seen how this analogy might be further extended to
the case where $\{x_{nt}\}$ is endogenous. Heuristically, estimation
of $m_{n}$ must be possible, in the integrated case, because the
joint dependence between $u_{t}$ and $x_{nt}$ becomes progressively
weaker, as $n,t\goesto\infty$. Although it is not immediately clear
how this notion should be made precise -- let alone what would be
a suitable analogue of it in the stationary setting -- the `location
shift' model of \citet{PS11EJ} may be counted as an important effort
in this direction.

\section{Preliminaries\label{sec:prelim}}

Preliminary to the proofs of our main results, this section collects
some auxiliary lemmas, proofs of which are given in \appref{auxproofs}
of the Supplement. We shall rely heavily on the use of the inverse
Fourier transform to analyse objects of the form $\expect_{t}f(x_{t+k})$,
similarly to \citet{BI95Stek}, \citet{Jeg04,Jeg08CDFP} and \citet{WP09Ecta,WP11ET}.
The following result permits the use of the `usual' inversion formula,
even in cases where $\hat{f}\notin L^{1}$, for $\hat{f}(\lambda)\defeq\int f(x)\e^{\i\lambda x}\diff x$.
\begin{lem}
\label{lem:fourinv} Suppose $Y=Y_{1}+Y_{2}$, where $Y_{1}$ is independent
of $(Y_{2},Z)$, and $Y_{i}$ has integrable characteristic function
$\psi_{Y_{i}}$. Then, for every $f\in\BI$, $y_{0}\in\reals$, and
$\expect\smlabs{g(Z)}<\infty$,
\begin{equation}
\expect f(y_{0}+Y)g(Z)=\frac{1}{2\pi}\int_{\reals}\hat{f}(\lambda)\e^{-\i\lambda y_{0}}\expect[\e^{-\i\lambda Y}g(Z)]\diff\lambda.\label{eq:fourinv}
\end{equation}

\end{lem}

Let $\filt_{s}^{t}\defeq\sigma(\{\epsilon_{r}\}_{r=s}^{t})$, noting
that $\filt_{s_{1}}^{s_{2}}$ and $\filt_{s_{3}}^{s_{4}}$ are independent
whenever $s_{1}\leq s_{2}<s_{3}\leq s_{4}$. We shall have frequent
recourse to the following decomposition,
\begin{align}
x_{t}=\sum_{k=1}^{t}v_{t}=\sum_{k=1}^{t}\sum_{l=0}^{\infty}\phi_{l}\epsilon_{k-l} & =\left[\sum_{i=0}^{\infty}\epsilon_{-i}\sum_{j=i+1}^{i+t}\phi_{j}+\sum_{i=t-s+1}^{t-1}\epsilon_{t-i}\sum_{j=0}^{i}\phi_{j}\right]+\sum_{i=0}^{t-s}\epsilon_{t-i}\sum_{j=0}^{i}\phi_{j}\nonumber \\
 & \eqdef x_{s-1,t}^{\ast}+x_{s,t,t}^{\prime},\label{eq:decomp1}
\end{align}
for $1\leq s\leq t$, where $x_{s-1,t}^{\ast}$ and $x_{s,t,t}^{\prime}$
are independent, and $x_{s-1,t}^{\ast}$ is $\filt_{-\infty}^{s-1}$-measurable.
Defining $a_{i}\defeq\sum_{j=0}^{i}\phi_{j}$, we may further decompose
$x_{s,t,t}^{\prime}$ as
\begin{equation}
x_{s,t,t}^{\prime}=\sum_{i=s}^{t}a_{t-i}\epsilon_{i}=\sum_{i=s}^{r}a_{t-i}\epsilon_{i}+\sum_{i=r+1}^{t}a_{t-i}\epsilon_{i}\eqdef x_{s,r,t}^{\prime}+x_{r+1,t,t}^{\prime},\label{eq:decomp2}
\end{equation}
where $x_{s,r,t}^{\prime}$ is $\filt_{s}^{r}$-measurable, and $x_{r+1,t,t}^{\prime}$
is $\filt_{r+1}^{t}$-measurable. The following property of the coefficients
$\{a_{i}\}$ is particularly important: there exist $0<\blw a\leq\abv a<\infty$,
and a $\minidx\in\naturals$ such that
\begin{equation}
\blw a\leq\inf_{\minidx+1\leq k}\inf_{\smlfloor{k/2}\leq l\le k}c_{k}^{-1}\smlabs{a_{l}}\leq\sup_{\minidx+1\leq k}\sup_{\smlfloor{k/2}\leq l\le k}c_{k}^{-1}\smlabs{a_{l}}\leq\abv a.\label{eq:coefcf}
\end{equation}
This is an easy consequence of Karamata's theorem (see \XREFlemmasproof{}
of the Supplement to \citet{Duffy14uniform} for a proof). Throughout
the remainder of the paper, $\minidx$ refers to the object of \eqref{coefcf};
it is also implicitly maintained $\minidx\geq8\cfidx$ for $\cfidx$
as in \assref{reg}\enuref{iidseq}.

Having decomposed $x_{t}$ into a sum of independent components, we
shall proceed to control such objects as the right side of \eqref{fourinv}
with the aid of the following lemma, which provides bounds on integrals
involving the characteristic functions of some of those components
of $x_{t}$. (This lemma summarises and refines some of the calculations
presented on pp.~15--21 of \citealp{Jeg08CDFP}.) In order to state
this result, we first note that \assref{reg}\enuref{iidseq} is equivalent
to the statement that 
\begin{equation}
\log\psi(\lambda)=-\smlabs{\lambda}^{\alpha}G(\lambda)\left[1-\i\beta\tan\left(\frac{\pi\alpha}{2}\right)\right]\label{eq:dofattr}
\end{equation}
for all $\lambda$ in a neighbourhood of the origin, where $G$ is
slowly varying at zero (see \citealp[Thm.~2.6.5]{IL71}). (Here, as
throughout the remainder of this paper, a slowly varying (or regularly
varying) function is understood to take only strictly positive values,
and have the property that $G(\lambda)=G(\smlabs{\lambda})$ for every
$\lambda\in\reals$.)
\begin{lem}
\label{lem:covbnd} Let $p\in[0,5]$, $q\in[1,2]$, and $z_{1},z_{2}\in\reals_{+}$.
Then
\begin{enumerate}
\item \label{enu:cov1}there exists a $\gamma_{1}>0$ such that for every
$t\geq0$, $k\geq\minidx+1$ and $m\in\{0,\ldots k-1\}$, 
\begin{align}
\int_{\reals}(z_{1}\smlabs{\lambda}^{p}\pmin z_{2})\smlabs{\expect\eta_{t+k-m}\e^{-\i\lambda x_{t+1,t+k,t+k}^{\prime}}}^{q}\diff\lambda & \lesssim z_{1}c_{m}^{q}d_{k}^{-(1+p+q)}+z_{2}\e^{-\gamma_{1}k};\label{eq:cov1a}
\end{align}
and if $F(u)\asymp G^{p/\alpha}(u)$ as $u\goesto0$, 
\begin{multline}
\int_{\reals}(z_{1}\smlabs{a_{k}}^{p}\smlabs{\lambda}^{p}F(a_{k}\lambda)\pmin z_{2})\smlabs{\expect\eta_{t+k-m}\e^{-\i\lambda x_{t+1,t+k,t+k}^{\prime}}}^{q}\diff\lambda\\
\lesssim z_{1}c_{m}^{q}k^{-p/\alpha}d_{k}^{-(1+q)}+z_{2}\e^{-\gamma_{1}k};\label{eq:cov1b}
\end{multline}

\item \label{enu:cov2}for every $t\geq1$, $k\geq\minidx+1$ and $s\in\{\minidx+1,\ldots,t\}$,
\[
\int_{\reals}\smlabs{\expect\e^{-\i\lambda x_{t-s+1,t-1,t+k}^{\prime}}}\diff\lambda\lesssim\frac{c_{s}}{c_{k+s}}d_{s}^{-1}.
\]

\end{enumerate}
\end{lem}

We note here, for future reference, that the preceding continues to
hold when $\eta_{t}$ is replaced by $\eta_{t}^{(\leq)}$ as defined
in \eqref{etatrunc} below. Let $\expect_{t}[\cdot]\defeq\expect[\cdot\mid\filt_{-\infty}^{t}]$.
The following is an easy consequence of the preceding.
\begin{lem}
\label{lem:1stmom}Suppose $f\in\BI$. Then
\begin{enumerate}
\item \label{enu:1stmom} for every $t\geq0$ and $k\geq\minidx+1$, 
\[
\expect_{t}\smlabs{f(x_{t+k})\eta_{t+k-m}}\lesssim d_{k}^{-1}\smlnorm f_{1}\cdot\begin{cases}
1 & \text{if }m\in\{0,\ldots,k-1\},\\
\smlabs{\eta_{t+k-m}} & \text{if }m\geq k;
\end{cases}
\]

\item \label{enu:1stcov}and, if in addition $m\in\{0,\ldots,k-1\}$,
\[
\smlabs{\expect_{t}f(x_{t+k})\eta_{t+k-m}}\lesssim c_{m}d_{k}^{-2}\smlnorm f_{1}.
\]

\end{enumerate}
\end{lem}

Recall the definitions of $\{d_{k}\}$ and $\{\nseq_{k}\}$ given
in \eqref{dkek} above. The following is a straightforward consequence
of Karamata's theorem.
\begin{lem}
\label{lem:summable}~
\begin{enumerate}
\item \label{enu:dm2} $\sum_{t=1}^{n}d_{t}^{-2}\lesssim\nseq_{n}^{1/2}$;
\item \label{enu:dna} $\sum_{k=1}^{\infty}k^{-1/2}d_{k}^{-3/2}<\infty$;
\item \label{enu:dth} $\sum_{m=0}^{\infty}\smlabs{\theta_{m}}(c_{m}+m^{1/2}\nseq_{m})<\infty$.
\end{enumerate}
\end{lem}

For the reader's convenience, Lemmas~\XREFcfprod{} and \XREFmgax{}
from \citet{Duffy14uniform} are reproduced below; see that paper
for the proofs. For the first of these, define
\[
\vartheta(z_{1},z_{2})\defeq\expect\left[\e^{-\i z_{1}\epsilon_{0}}-\expect\e^{-\i z_{1}\epsilon_{0}}\right]\left[\e^{-\i z_{2}\epsilon_{0}}-\expect\e^{-\i z_{2}\epsilon_{0}}\right].
\]

\begin{lem}
\label{lem:cfdiffprod}Uniformly over $z_{1},z_{2}\in\reals$, 
\[
\smlabs{\vartheta(z_{1},z_{2})}\lesssim[\smlabs{z_{1}}^{\alpha}\tilde{G}(z_{1})\pmin1]^{1/2}[\smlabs{z_{2}}^{\alpha}\tilde{G}(z_{2})\pmin1]^{1/2}
\]
where $\tilde{G}(u)\asymp G(u)$ as $u\goesto0$.
\end{lem}

Let $\smlnorm{\cdot}_{\tau_{1}}$ denote the Orlicz norm associated
to $\tau_{1}(x)\defeq\e^{x}-1$. (See \citealp{VVW96}, p.~95 for
the definition of an Orlicz norm.) For a martingale $M\defeq\{M_{t}\}_{t=0}^{n}$
with associated filtration $\filtg\defeq\{\filtg_{t}\}_{t=0}^{n}$,
define 
\begin{align}
\smlqv M & \defeq\sum_{t=1}^{n}(M_{t}-M_{t-1})^{2} & \smlcv M & \defeq\sum_{t=1}^{n}\expect[(M_{t}-M_{t-1})^{2}\mid\filtg_{t-1}].\label{eq:qvcv}
\end{align}
We say that $M$ is \emph{initialised at zero} if $M_{0}=0$. The
next result is a straightforward consequence of Theorem~2.1 in \citet{BT08AAP}.
\begin{lem}
\label{lem:mgmax} Let $\{\Theta_{n}\}$ denote a sequence of index
sets, and $\{K_{n}\}$ a real sequence such that $\#\Theta_{n}+K_{n}\lesssim n^{C}$.
Suppose that for each $n\in\naturals$, $k\in\{1,\ldots,K_{n}\}$
and $\theta\in\Theta_{n}$, $M_{nk}(\theta)$ is a martingale, initialised
at zero, for which
\begin{equation}
\omega_{nk}^{2}\defeq\max_{\theta\in\Theta_{n}}\{\smlnorm{\smlqv{M_{nk}(\theta)}}_{\tau_{1}}\pmax\smlnorm{\smlcv{M_{nk}(\theta)}}_{\tau_{1}}\}<\infty.\label{eq:omnk}
\end{equation}
Then 
\[
\max_{\theta\in\Theta_{n}}\abs{\sum_{k=1}^{K_{n}}M_{nk}(\theta)}\isbigop\left(\sum_{k=1}^{K_{n}}\omega_{nk}\right)\log n.
\]

\end{lem}

\section{Controlling the truncated covariance process\label{sec:covproof}}

We turn first to the proof of \propref{covtruncord}. For this section
only, we shall denote $\eta_{t}^{(\leq)}$ by simply $\eta_{t}$.
Then recalling \eqref{Snf} above, we may write
\begin{equation}
\S_{n}f=\sum_{t=1}^{n}f(x_{t})u_{t}^{(\leq)}=\sum_{m=0}^{n}\theta_{m}\sum_{t=1}^{n}f(x_{t})\eta_{t-m}\eqdef\sum_{m=0}^{n}\theta_{m}\S_{nm}f.\label{eq:Sndecomp}
\end{equation}
For each $m\in\{0,\ldots,n\}$, by following a procedure identical
to that described in \XREFptwise{} in \citet{Duffy14uniform}, the
process $\S_{nm}f$ may be decomposed as
\begin{equation}
\S_{nm}f=\N_{nm}f+\sum_{k=0}^{n-1}\M_{nmk}f\label{eq:Snmdecomp}
\end{equation}
where
\begin{align*}
\N_{nm}f & \defeq\sum_{t=1}^{n}\expect_{0}f(x_{t})\eta_{t-m} & \M_{nmk}f & \defeq\sum_{t=1}^{n-k}\mg_{mkt}f
\end{align*}
\begin{equation}
\mg_{mkt}f\defeq\expect_{t}f(x_{t+k})\eta_{t+k-m}-\expect_{t-1}f(x_{t+k})\eta_{t+k-m},\label{eq:xi}
\end{equation}
and we have defined $\expect_{t}[\cdot]\defeq\expect[\cdot\mid\filt_{-\infty}^{t}]$.

A suitable bound for $\smlnorm{\N_{nm}f}_{\infty}$ is provided by
\lemref{1stmom}\enuref{1stcov}. By construction, $\{\xi_{mkt},\filt_{-\infty}^{t}\}_{t=1}^{n-k}$
forms a martingale difference sequence for each $(m,k)$, and so control
over each of the martingale `pieces' $\M_{nmk}f$ may be obtained
via control over
\begin{gather*}
\U_{nmk}f\defeq\smlqv{\M_{nmk}f}=\sum_{t=1}^{n-k}\mg_{mkt}^{2}f\\
\V_{nmk}f\defeq\smlcv{\M_{nmk}f}=\sum_{t=1}^{n-k}\expect_{t-1}\mg_{mkt}^{2}f,
\end{gather*}
in combination with \lemref{mgmax}. Defining
\begin{equation}
\varsigma_{nm}(f)\defeq\truncnorm{\eta}\smlnorm f_{\infty}+(\nseq_{m}\truncnorm{\eta}+c_{m}\nseq_{n}^{1/2})\smlnorm f_{1}\label{eq:varsig}
\end{equation}
and
\begin{equation}
\sigma_{nmk}^{2}(f)\defeq\begin{cases}
\truncnorm{\eta}^{2}\smlnorm f_{\infty}^{2}+(\nseq_{m}\truncnorm{\eta}^{2}+\nseq_{n})\smlnorm f_{2}^{2} & \text{if }k\in\{0,\ldots,\minidx\}\\
d_{k}^{-1}(\nseq_{m}\truncnorm{\eta}^{2}+\nseq_{n})\smlnorm f_{1}^{2} & \text{if }k\in\{\minidx+1,\ldots,m\}\\
(k^{-1}d_{k}^{-3}c_{m}^{2}+\e^{-\gamma_{1}k})\nseq_{n}\smlnorm f_{1}^{2} & \text{if }k\in\{\minidx\pmax m+1,\ldots,n-1\},
\end{cases}\label{eq:signkdef}
\end{equation}
our first result is
\begin{lem}
\label{lem:fundamental}For all $m\in\{0,\ldots,n\}$
\begin{equation}
\smlnorm{\N_{nm}f}_{\infty}\lesssim\varsigma_{nm}(f),\label{eq:Nnbnd}
\end{equation}
and all $0\leq k\leq n-1$,
\begin{equation}
\smlnorm{\U_{nmk}f}_{\tau_{1}}\pmax\smlnorm{\V_{nmk}f}_{\tau_{1}}\lesssim\sigma_{nmk}^{2}(f).\label{eq:sqbnds}
\end{equation}

\end{lem}

The proof of \eqref{sqbnds}, in turn, relies upon
\begin{lem}
\label{lem:mgbnd}For every $m\in\{0,\ldots,n\}$, $k\in\{0,\ldots,n-1\}$
and $t\in\{1,\ldots,n-k\}$,
\[
\smlnorm{\mg_{mkt}^{2}f}_{\infty}+\sum_{s=1}^{n-k-t}\smlnorm{\expect_{t}\mg_{m,k,t+s}^{2}f}_{\infty}\lesssim\sigma_{nmk}^{2}(f).
\]

\end{lem}

For the next result, recall the definition of $\delta_{n}(\Gset)$
given in \eqref{delnf} above.
\begin{lem}
\label{lem:dnsum} If $\Gset\subset\BI$, then
\[
\sum_{m=0}^{n}\smlabs{\theta_{m}}\sup_{f\in\Gset}\varsigma_{nm}(f)+\sum_{m=0}^{n}\sum_{k=0}^{n-1}\smlabs{\theta_{m}}\sup_{f\in\Gset}\sigma_{nmk}(f)\lesssim\delta_{n}(\Gset).
\]

\end{lem}

The proofs of these results are given below. We first turn to the
\begin{proof}[Proof of \propref{covtruncord}]
 The proof is almost identical to the proof of \XREFmaxSn{} in \citet{Duffy14uniform}.
In view of \lemref{fundamental} and \lemref{dnsum}, we have immediately
that
\[
\max_{f\in\Fset_{n}}\abs{\sum_{m=0}^{n}\theta_{m}\N_{nm}f}\lesssim\sum_{m=0}^{n}\smlabs{\theta_{m}}\max_{f\in\Fset_{n}}\varsigma_{nm}(f)\lesssim\delta_{n}(\Fset_{n}),
\]
and through an application of \lemref{mgmax}, that 
\[
\max_{f\in\Fset_{n}}\abs{\sum_{m=0}^{n}\sum_{k=0}^{n-1}\theta_{m}\M_{nmk}f}\isbigop\delta_{n}(\Fset_{n})\log n
\]
whence the result follows from \eqref{Sndecomp} and \eqref{Snmdecomp}.
\end{proof}
 
\begin{proof}[Proof of \lemref{fundamental}]
 For \eqref{Nnbnd}, note that by \lemref{1stmom},
\[
\smlnorm{\expect_{0}f(x_{t})\eta_{t-m}}_{\infty}\lesssim\begin{cases}
\truncnorm{\eta}\smlnorm f_{\infty} & \text{if }t\in\{1,\ldots,\minidx\}\\
d_{t}^{-1}\truncnorm{\eta}\smlnorm f_{1} & \text{if }t\in\{\minidx+1,\ldots,m\}\\
d_{t}^{-2}c_{m}\smlnorm f_{1} & \text{if }t\in\{\minidx\pmax m+1,\ldots n\},
\end{cases}
\]
whence, by Karamata's theorem and \lemref{summable}\enuref{dm2},
\begin{align*}
\smlnorm{\N_{nm}f}_{\infty} & \leq\left(\sum_{t=1}^{\minidx}+\sum_{t=\minidx+1}^{m}+\sum_{t=\minidx\pmax m+1}^{n}\right)\smlnorm{\expect_{0}f(x_{t})\eta_{t-m}}_{\infty}\\
 & \lesssim\minidx\truncnorm{\eta}\smlnorm f_{\infty}+\left(\truncnorm{\eta}\sum_{t=\minidx+1}^{m}d_{t}^{-1}+c_{m}\sum_{t=\minidx\pmax m+1}^{n}d_{t}^{-2}\right)\smlnorm f_{1}\\
 & \lesssim\truncnorm{\eta}\smlnorm f_{\infty}+\left[\nseq_{m}\truncnorm{\eta}+c_{m}\nseq_{n}^{1/2}\right]\smlnorm f_{1}.
\end{align*}
\eqref{sqbnds} follows from \lemref{mgbnd} in exactly the manner
described in the proof of \XREFfundamental{} in \citet{Duffy14uniform}.
\end{proof}
 
\begin{proof}[Proof of \lemref{mgbnd}]
 The proof is similar to that of \XREFmgbnd{} in \citet{Duffy14uniform}.
Let $\minidxm\defeq\minidx\pmax m$. We shall obtain the requisite
bound for $\expect_{t}\mg_{m,k,t+s}^{2}f$ by providing a bound for
$\expect_{t-s}\mg_{mkt}^{2}f$ (for $s\in\{1,\ldots,t\}$) that depends
only on $m$, $k$ and $s$ (and \emph{not} $t$), separately considering
the cases where
\begin{enumerate}
\item $k\in\{\minidxm+1,\ldots,n-t\}$;
\item $k\in\{\minidx,\ldots,\minidxm\}$; and
\item $k\in\{0,\ldots,\minidx\}$.
\end{enumerate}

\paragraph*{(i)}

Recall the decomposition given in \eqref{decomp1} and \eqref{decomp2}
above, applied here to reduce $x_{t+k}$ to a sum of independent pieces,
\begin{align*}
x_{t+k} & =x_{0,t+k}^{\ast}+x_{1,t-1,t+k}^{\prime}+x_{t,t,t+k}^{\prime}+x_{t+1,t+k,t+k}^{\prime}\\
 & =x_{0,t+k}^{\ast}+x_{1,t-1,t+k}^{\prime}+a_{k}\epsilon_{t}+x_{t+1,t+k,t+k}^{\prime}
\end{align*}
with the convention that $x_{1,t-1,t+k}^{\prime}=0$ if $t=1$, so
that by Fourier inversion (\lemref{fourinv}),
\begin{align}
\mg_{mkt}f & =\expect_{t}f(x_{t+k})\eta_{t+k-m}-\expect_{t-1}f(x_{t+k})\eta_{t+k-m}\nonumber \\
 & =\frac{1}{2\pi}\int\hat{f}(\lambda)\e^{-\i\lambda x_{0,t+k}^{\ast}}\e^{-\i\lambda x_{1,t-1,t+k}^{\prime}}\label{eq:0enxikt}\\
 & \qquad\qquad\cdot\left[\e^{-\i\lambda a_{k}\epsilon_{t}}-\expect\e^{-\i\lambda a_{k}\epsilon_{t}}\right]\expect\eta_{t+k-m}\e^{-\i\lambda x_{t+1,t+k,t+k}^{\prime}}\diff\lambda.\nonumber 
\end{align}
Thence
\begin{align}
\mg_{mkt}^{2}f & =\frac{1}{(2\pi)^{2}}\iint\hat{f}(\lambda_{1})\hat{f}(\lambda_{2})\e^{-\i(\lambda_{1}+\lambda_{2})x_{0,t+k}^{\ast}}\e^{-\i(\lambda_{1}+\lambda_{2})x_{1,t-1,t+k}^{\prime}}\label{eq:0enxiktsq}\\
 & \qquad\qquad\qquad\cdot\left[\e^{-\i\lambda_{1}a_{k}\epsilon_{t}}-\expect\e^{-\i\lambda_{1}a_{k}\epsilon_{t}}\right]\left[\e^{-\i\lambda_{2}a_{k}\epsilon_{t}}-\expect\e^{-\i\lambda_{2}a_{k}\epsilon_{t}}\right]\nonumber \\
 & \qquad\qquad\qquad\cdot\expect\eta_{t+k-m}\e^{-\i\lambda_{1}x_{t+1,t+k,t+k}^{\prime}}\expect\eta_{t+k-m}\e^{-\i\lambda_{2}x_{t+1,t+k,t+k}^{\prime}}\diff\lambda_{1}\diff\lambda_{2}.\nonumber 
\end{align}
(Note that $\mg_{mkt}f$ is real-valued, so $\mg_{mkt}^{2}f=\smlabs{\mg_{mkt}f}^{2}=\mg_{mkt}f\cdot\abv{\mg_{mkt}f}=\mg_{mkt}f\cdot\mg_{mkt}f$.)

Now suppose $s\in\{k+1,\ldots,t\}$. Taking conditional expectations
on both sides of \eqref{0enxiktsq} gives
\begin{align*}
\expect_{t-s}\mg_{mkt}^{2}f & =\frac{1}{(2\pi)^{2}}\iint\hat{f}(\lambda_{1})\hat{f}(\lambda_{2})\e^{-\i(\lambda_{1}+\lambda_{2})x_{0,t+k}^{\ast}}\e^{-\i(\lambda_{1}+\lambda_{2})x_{1,t-s,t+k}^{\prime}}\\
 & \qquad\qquad\qquad\cdot\vartheta(\lambda_{1}a_{k},\lambda_{2}a_{k})\expect\e^{-\i(\lambda_{1}+\lambda_{2})x_{t-s+1,t-1,t+k}^{\prime}}\\
 & \qquad\qquad\qquad\cdot\expect\eta_{t+k-m}\e^{-\i\lambda_{1}x_{t+1,t+k,t+k}^{\prime}}\expect\eta_{t+k-m}\e^{-\i\lambda_{2}x_{t+1,t+k,t+k}^{\prime}}\diff\lambda_{1}\diff\lambda_{2},
\end{align*}
where we have defined
\[
\vartheta(z_{1},z_{2})\defeq\expect\left[\e^{-\i z_{1}\epsilon_{0}}-\expect\e^{-\i z_{1}\epsilon_{0}}\right]\left[\e^{-\i z_{1}\epsilon_{0}}-\expect\e^{-\i z_{2}\epsilon_{0}}\right]
\]
for $z_{1},z_{2}\in\reals$, and made the further decomposition
\[
x_{1,t-1,t+k}^{\prime}=x_{1,t-s,t+k}^{\prime}+x_{t-s+1,t-1,t+k}^{\prime}
\]
with the convention that $x_{1,t-s,t+k}^{\prime}=0$ if $s=t$. Thence,
using \eqref{coefcf} and \lemref{cfdiffprod}, and the inequalities
$\smlabs{\hat{f}(\lambda)}\leq\smlnorm f_{1}$ and $\smlabs{ab}\lesssim\smlabs a^{2}+\smlabs b^{2}$,
we obtain
\begin{align}
\expect_{t-s}\mg_{mkt}^{2}f & \lesssim\iint\smlabs{\hat{f}(\lambda_{1})\hat{f}(\lambda_{2})}\smlabs{\expect\e^{-\i(\lambda_{1}+\lambda_{2})x_{t-s+1,t-1,t+k}^{\prime}}}\label{eq:fors0case}\\
 & \qquad\qquad\cdot(\smlabs{a_{k}\lambda_{1}}^{\alpha}\tilde{G}(a_{k}\lambda_{1})\pmin1)^{1/2}(\smlabs{a_{k}\lambda_{2}}^{\alpha}\tilde{G}(a_{k}\lambda_{1})\pmin1)^{1/2}\nonumber \\
 & \qquad\qquad\cdot\smlabs{\expect\eta_{t+k-m}\e^{-\i\lambda_{1}x_{t+1,t+k,t+k}^{\prime}}}\smlabs{\expect\eta_{t+k-m}\e^{-\i\lambda_{2}x_{t+1,t+k,t+k}^{\prime}}}\diff\lambda_{1}\diff\lambda_{2}\nonumber \\
 & \lesssim\smlnorm f_{1}^{2}\int(\smlabs{a_{k}\lambda_{1}}^{\alpha}\tilde{G}(a_{k}\lambda_{1})\pmin1)\smlabs{\expect\eta_{t+k-m}\e^{-\i\lambda_{1}x_{t+1,t+k,t+k}^{\prime}}}^{2}\label{eq:0enxikt1}\\
 & \qquad\qquad\qquad\int\smlabs{\expect\e^{-\i(\lambda_{1}+\lambda_{2})x_{t-s+1,t-1,t+k}^{\prime}}}\diff\lambda_{2}\diff\lambda_{1},\nonumber 
\end{align}
where we have appealed to symmetry (in $\lambda_{1}$ and $\lambda_{2}$)
to reduce the final bound to a single term. By a change of variables
and \lemref{covbnd}\enuref{cov2}, 
\begin{align}
\int\smlabs{\expect\e^{-\i(\lambda_{1}+\lambda_{2})x_{t-s+1,t-1,t+k}^{\prime}}}\diff\lambda_{2} & =\int\smlabs{\expect\e^{-\i\lambda x_{t-s+1,t-1,t+k}^{\prime}}}\diff\lambda\lesssim\frac{c_{s}}{c_{k+s}}d_{s}^{-1},\label{eq:0enxikt2}
\end{align}
while \lemref{covbnd}\enuref{cov1} gives 
\begin{align}
\int(\smlabs{a_{k}\lambda_{1}}^{\alpha}\tilde{G}(a_{k}\lambda_{1})\pmin1)\smlabs{\expect\eta_{t+k-m}\e^{-\i\lambda_{1}x_{t+1,t+k,t+k}^{\prime}}}^{2}\diff\lambda & \lesssim k^{-1}d_{k}^{-3}c_{m}^{2}+\e^{-\gamma_{1}k}\label{eq:0enxikt3}
\end{align}
Together, \eqref{0enxikt1}--\eqref{0enxikt3}\noeqref{eq:0enxikt2}
yield
\begin{align}
\expect_{t-s}\mg_{mkt}^{2}f & \lesssim\frac{c_{s}}{c_{k+s}}d_{s}^{-1}(k^{-1}d_{k}^{-3}c_{m}^{2}+\e^{-\gamma_{1}k})\smlnorm f_{1}^{2}\label{eq:0enmkbnd1}
\end{align}

When $s\in\{1,\ldots,k\}$, \eqref{fors0case} continues to hold,
whence
\begin{align}
\expect_{t-s}\mg_{mkt}^{2}f & \lesssim\biggl(\int\smlabs{\hat{f}(\lambda)}(\smlabs{a_{k}\lambda_{1}}^{\alpha}\tilde{G}(a_{k}\lambda_{1})\pmin1)\smlabs{\expect\eta_{t+k-m}\e^{-\i\lambda x_{t+1,t+k,t+k}^{\prime}}}\diff\lambda\biggr)^{2}\nonumber \\
 & \lesssim\smlnorm f_{1}^{2}(k^{-1/2}d_{k}^{-2}c_{m}+\e^{-\gamma_{1}k})^{2}\nonumber \\
 & \lesssim d_{s}^{-1}(k^{-1}d_{k}^{-3}c_{m}^{2}+\e^{-\gamma_{1}k})\smlnorm f_{1}^{2}\label{eq:0enmkbnd1b}
\end{align}
by \lemref{covbnd}\enuref{cov1}; the replacement of a $d_{k}^{-1}$
by $d_{s}^{-1}$ in the final bound is justified because $s\leq k$.
Since $\{c_{k}\}$ is regularly varying and $k\geq\minidx+1$, it
follows from Potter's inequality \citep[Thm.~1.5.6(iii)]{Bingham87}
that
\begin{equation}
\sum_{s=1}^{k}d_{s}^{-1}+\sum_{s=k+1}^{n}\frac{c_{s}}{c_{k+s}}d_{s}^{-1}\lesssim\sum_{s=1}^{n}d_{s}^{-1}\lesssim nd_{n}^{-1}=\nseq_{n},\label{eq:0enssum}
\end{equation}
with the final bound following by Karamata's theorem. As noted above,
since the bounds \eqref{0enmkbnd1} and \eqref{0enmkbnd1b} do not
depend on $t$, they apply also to $\expect_{t}\mg_{m,k,t+s}^{2}f$.
Hence, in view of the preceding,
\begin{align*}
\sum_{s=1}^{n-k-t}\expect_{t}\mg_{m,k,t+s}^{2}f & \lesssim(k^{-1}d_{k}^{-3}c_{m}^{2}+\e^{-\gamma_{1}k})\smlnorm f_{1}^{2}\left[\sum_{s=1}^{k}d_{s}^{-1}+\sum_{s=k+1}^{n-k-t}\frac{c_{s}}{c_{k+s}}d_{s}^{-1}\right]\\
 & \lesssim(k^{-1}d_{k}^{-3}c_{m}^{2}+\e^{-\gamma_{1}k})\nseq_{n}\smlnorm f_{1}^{2}.
\end{align*}

Turning now to $\smlnorm{\mg_{mkt}^{2}f}_{\infty}$, note that \eqref{0enxikt}
still holds, with the convention that $x_{1,t-1,t+k}=0$ if $t=1$.
Thus, again by \lemref{covbnd}\enuref{cov1},
\begin{align}
\smlnorm{\mg_{mkt}^{2}f}_{\infty} & \lesssim\left(\int\smlabs{\hat{f}(\lambda)}\smlabs{\expect\eta_{t+k-m}\e^{-\i\lambda x_{t+1,t+k,t+k}^{\prime}}}\diff\lambda\right)^{2}\label{eq:0enxisupbnd}\\
 & \leq\smlnorm f_{1}^{2}\left(\int\smlabs{\expect\eta_{t+k-m}\e^{-\i\lambda x_{t+1,t+k,t+k}^{\prime}}}\diff\lambda\right)^{2}\nonumber \\
 & \lesssim(d_{k}^{-2}c_{m}+\e^{-\gamma_{1}k})^{2}\smlnorm f_{1}^{2}\nonumber \\
 & \lesssim(k^{-1}d_{k}^{-3}c_{m}^{2}+\e^{-\gamma_{1}k})\nseq_{n}\smlnorm f_{1}^{2};\label{eq:0enmkbnd2}
\end{align}
where the final bound follows because $k\leq n$, and so $d_{k}^{-1}\lesssim k^{-1}\nseq_{n}$.

\paragraph*{(ii)}

Note that $\{\minidx+1,\ldots,\minidxm\}$ can only be nonempty if
$\minidx<m=\minidxm$; thus $k\leq m$ in this case. Using $\smlabs{a+b}^{2}\lesssim\smlabs a^{2}+\smlabs b^{2}$
gives
\begin{align}
\expect_{t-s}\mg_{mkt}^{2}f & \lesssim\expect_{t-s}\left[\left(\expect_{t}\smlabs{f(x_{t+k})\eta_{t+k-m}}\right)^{2}+\left(\expect_{t-1}\smlabs{f(x_{t+k})\eta_{t+k-m}}\right)^{2}\right].\label{eq:k0m0p1}
\end{align}
Suppose $s\in\{m-k+1,\ldots,t\}$. Then by successive applications
of \lemref{1stmom}\enuref{1stmom},
\[
\expect_{t-s}\left(\expect_{t}\smlabs{f(x_{t+k})\eta_{t+k-m}}\right)^{2}\lesssim d_{s}^{-1}d_{k}^{-1}\smlnorm f_{1}^{2},
\]
and similarly for the second term on the right side of \eqref{k0m0p1}.
Thus
\begin{equation}
\expect_{t-s}\mg_{mkt}^{2}f\lesssim d_{s}^{-1}d_{k}^{-1}\smlnorm f_{1}^{2}.\label{eq:k0m0b1}
\end{equation}
When $s\in\{1,\ldots,m-k\}$, further applications of \lemref{1stmom}\enuref{1stmom}
give
\begin{align*}
\expect_{t-s}\left(\expect_{t}\smlabs{f(x_{t+k})\eta_{t+k-m}}\right)^{2} & \leq\truncnorm{\eta}^{2}\expect_{t-s}\left(\expect_{t}\smlabs{f(x_{t+k})}\right)^{2}\\
 & \lesssim d_{s}^{-1}d_{k}^{-1}\truncnorm{\eta}^{2}\smlnorm f_{1}^{2},
\end{align*}
whence
\begin{equation}
\expect_{t-s}\mg_{mkt}\lesssim d_{s}^{-1}d_{k}^{-1}\truncnorm{\eta}^{2}\smlnorm f_{1}^{2}.\label{eq:k0m0b2}
\end{equation}
Together, \eqref{k0m0b1} and \eqref{k0m0b2} give
\begin{align*}
\sum_{s=1}^{n-k-t}\expect_{t}\mg_{m,k,t+s}^{2}f & =\sum_{s=1}^{m-k}\expect_{t}\mg_{m,k,t+s}^{2}f+\sum_{s=m-k+1}^{n-k-t}\expect_{t}\mg_{m,k,t+s}^{2}f\\
 & \lesssim d_{k}^{-1}\left[\truncnorm{\eta}^{2}\sum_{s=1}^{m-k}d_{s}^{-1}+\sum_{s=m-k+1}^{n-k-t}d_{s}^{-1}\right]\smlnorm f_{1}^{2}\\
 & \lesssim d_{k}^{-1}\left[\nseq_{m}\truncnorm{\eta}^{2}+\nseq_{n}\right]\smlnorm f_{1}^{2}
\end{align*}
by Karamata's theorem.

Regarding $\smlnorm{\mg_{mkt}^{2}f}_{\infty}$, it follows from \lemref{1stmom}\enuref{1stmom}
that
\[
\left(\expect_{t}\smlabs{f(x_{t+k})\eta_{t+k-m}}\right)^{2}\leq\truncnorm{\eta}^{2}\left(\expect_{t}\smlabs{f(x_{t+k})}\right)^{2}\lesssim d_{k}^{-2}\truncnorm{\eta}^{2}\smlnorm f_{1}^{2}.
\]

\paragraph*{(iii)}

Suppose that $s\in\{(m-k)\pmax\minidx+1,\ldots,t\}$. Using $\smlabs{a+b}^{2}\lesssim\smlabs a^{2}+\smlabs b^{2}$,
Jensen's inequality and \lemref{1stmom}\enuref{1stmom}, we have
\[
\expect_{t-s}\mg_{mkt}^{2}f\lesssim\expect_{t-s}\smlabs{f(x_{t+k})\eta_{t+k-m}}^{2}\lesssim d_{s}^{-1}\smlnorm f_{2}^{2}.
\]
When $s\in\{\minidx+1,\ldots,m-k\}$, we have similarly
\[
\expect_{t-s}\mg_{mkt}^{2}f\lesssim\truncnorm{\eta}^{2}\expect_{t-s}\smlabs{f(x_{t+k})}^{2}\lesssim d_{s}^{-1}\truncnorm{\eta}^{2}\smlnorm f_{2}^{2},
\]
and for $s\in\{1,\ldots,\minidx\}$, we may use the crude bound
\begin{equation}
\expect_{t-s}\mg_{mkt}^{2}f\leq\smlnorm{\mg_{mkt}^{2}f}_{\infty}\leq\truncnorm{\eta}^{2}\smlnorm f_{\infty}^{2}.\label{eq:bndinf3}
\end{equation}
Hence
\begin{align*}
\sum_{s=1}^{n-k-t}\expect_{t}\mg_{m,k,t+s}^{2}f & =\left(\sum_{s=1}^{\minidx}+\sum_{s=\minidx+1}^{m-k}+\sum_{s=(m-k)\pmax\minidx+1}^{n-k-t}\right)\expect_{t}\mg_{m,k,t+s}^{2}f\\
 & \lesssim\minidx\truncnorm{\eta}^{2}\smlnorm f_{\infty}^{2}+\smlnorm f_{2}^{2}\left(\truncnorm{\eta}^{2}\sum_{s=\minidx+1}^{m-k}d_{s}^{-1}+\sum_{s=(m-k)\pmax\minidx+1}^{n-k-t}d_{s}^{-1}\right)\\
 & \lesssim\truncnorm{\eta}^{2}\smlnorm f_{\infty}^{2}+\left[\nseq_{m}\truncnorm{\eta}^{2}+\nseq_{n}\right]\smlnorm f_{2}^{2}
\end{align*}
by Karamata's theorem. The required bound for $\smlnorm{\mg_{mkt}^{2}f}_{\infty}$
is given in \eqref{bndinf3}.
\end{proof}
 
\begin{proof}[Proof of \lemref{dnsum}]
 It is evident from \eqref{varsig} and \lemref{summable}\enuref{dth}
that
\begin{align}
\sum_{m=0}^{n}\smlabs{\theta_{m}}\sup_{f\in\Gset}\varsigma_{nm}(f) & \leq\truncnorm{\eta}\smlnorm{\Gset}_{\infty}+\left[\truncnorm{\eta}\sum_{m=0}^{n}\smlabs{\theta_{m}}\nseq_{m}+\nseq_{n}^{1/2}\sum_{m=0}^{n}\smlabs{\theta_{m}}c_{m}\right]\smlnorm{\Gset}_{1}\nonumber \\
 & \lesssim\truncnorm{\eta}\smlnorm{\Gset}_{\infty}+\left[\truncnorm{\eta}+\nseq_{n}^{1/2}\right]\smlnorm{\Gset}_{1}\label{eq:dnsump1}
\end{align}
Considering the three parts of \eqref{signkdef} separately, we first
have
\begin{align}
\sum_{m=0}^{n}\sum_{k=0}^{\minidx}\smlabs{\theta_{m}}\sup_{f\in\Gset}\sigma_{nmk}(f) & \lesssim\truncnorm{\eta}\smlnorm{\Gset}_{\infty}+\left[\truncnorm{\eta}\sum_{m=0}^{n}\smlabs{\theta_{m}}\nseq_{m}^{1/2}+\nseq_{n}^{1/2}\right]\smlnorm{\Gset}_{2}\nonumber \\
 & \lesssim\truncnorm{\eta}\smlnorm{\Gset}_{\infty}+\left[\truncnorm{\eta}+\nseq_{n}^{1/2}\right]\smlnorm{\Gset}_{2}\label{eq:dnsump2}
\end{align}
since $\minidx$ is fixed and finite. Next,
\begin{align}
\sum_{m=0}^{n}\sum_{k=\minidx+1}^{m}\smlabs{\theta_{m}}\sup_{f\in\Gset}\sigma_{nmk}(f) & \lesssim\sum_{m=0}^{n}\smlabs{\theta_{m}}\left[\nseq_{m}^{1/2}\truncnorm{\eta}+\nseq_{n}^{1/2}\right]\smlnorm{\Gset}_{1}\sum_{k=\minidx+1}^{m}d_{k}^{-1/2}\nonumber \\
 & \lesssim\sum_{m=0}^{n}\smlabs{\theta_{m}}m^{1/2}\nseq_{m}^{1/2}\left[\nseq_{m}^{1/2}\truncnorm{\eta}+\nseq_{n}^{1/2}\right]\smlnorm{\Gset}_{1}\nonumber \\
 & \lesssim\left[\truncnorm{\eta}+\nseq_{n}^{1/2}\right]\smlnorm{\Gset}_{1},\label{eq:dnsump3}
\end{align}
using Karamata's theorem and \lemref{summable}\enuref{dth}. Finally,
\begin{align}
\sum_{m=0}^{n}\sum_{k=m+1}^{n}\smlabs{\theta_{m}}\sup_{f\in\Gset}\sigma_{nmk}(f) & \lesssim\sum_{m=0}^{n}\smlabs{\theta_{m}}c_{m}\nseq_{n}^{1/2}\smlnorm{\Gset}_{1}\sum_{k=m+1}^{n}k^{-1/2}d_{k}^{-3/2}\nonumber \\
 & \lesssim\nseq_{n}^{1/2}\smlnorm{\Gset}_{1}\label{eq:dnsump4}
\end{align}
by parts \enuref{dna} and \enuref{dth} of \lemref{summable}. Recalling
\eqref{delnf}, the result now follows from \eqref{dnsump1}--\eqref{dnsump4}.\noeqref{eq:dnsump2,eq:dnsump3}
\end{proof}

\section{Proofs of the order estimates\label{sec:orderest}}

The proof of \eqref{covorderest} in \thmref{orderest} may be broken
into three parts:
\begin{enumerate}[label=(\alph*)]
\item \label{enu:trunc}a truncation argument permits $\{u_{t}\}$ to be
replaced by $\{u_{t}^{(\leq)}\}$ on the left side of \eqref{covorderest};
\item \label{enu:reduct}the supremum over $(a,h)\in\reals\times\Hspc_{n}$
is reduced to a maximum over a (growing) finite set; and
\item \label{enu:finitemax}an application of \propref{covtruncord} yields
the requisite bound over this finite set.
\end{enumerate}
These steps are described below, following which we provide details
of the modifications necessary for the proof of \eqref{0enorderest}.

\subsection{Truncation\label{sub:truncation}}

Decomposing
\begin{align*}
u_{t} & =\sum_{k=0}^{\infty}\theta_{k}\eta_{t-k}=\sum_{k=0}^{n}\theta_{k}\eta_{t-k}+\sum_{k=n+1}^{\infty}\theta_{k}\eta_{t-k}\eqdef u_{t}^{(n)}+u_{t}^{(-n)}
\end{align*}
we have by the Cauchy-Schwarz inequality that 
\begin{multline*}
\sup_{a\in\reals}\abs{\frac{1}{(e_{n}h_{n})^{1/2}}\sum_{t=1}^{n}f\left(\frac{x_{t}-d_{n}a}{h_{n}}\right)u_{t}^{(-n)}}\\
\leq\left(\sup_{a\in\reals}\frac{1}{(e_{n}h_{n})^{1/2}}\sum_{t=1}^{n}f^{2}\left(\frac{x_{t}-d_{n}a}{h_{n}}\right)\right)^{1/2}\left(\sum_{t=1}^{n}(u_{t}^{(-n)})^{2}\right)^{1/2}=o_{p}(1)
\end{multline*}
since the first term on the RHS is $O_{p}(1)$ by \propref{generalIP},
and
\[
\expect\sum_{t=1}^{n}(u_{t}^{(-n)})^{2}\lesssim n\sum_{j=n+1}^{\infty}\theta_{j}^{2}\leq\sum_{j=n+1}^{\infty}\theta_{j}^{2}j\goesto0.
\]
Next, define for $-n\leq t\leq n$, 
\begin{align}
\eta_{t}^{(\leq)} & \defeq\eta_{t}\indic\{\smlabs{\eta_{t}}\leq n^{1/\etamom}\}-\expect\eta_{0}\indic\{\smlabs{\eta_{0}}\leq n^{1/\etamom}\}\label{eq:etatrunc}
\end{align}
and $\eta_{t}^{(>)}\defeq\eta_{t}-\eta_{t}^{(\leq)}.$ Then setting
\begin{align}
u_{t}^{(\leq)} & \defeq\sum_{k=0}^{n}\theta_{k}\eta_{t-k}^{(\leq)} & u_{t}^{(>)} & \defeq\sum_{k=0}^{n}\theta_{k}\eta_{t-k}^{(>)}\label{eq:utrunc}
\end{align}
we see that
\begin{align*}
\Prob\left\{ \sup_{a\in\reals}\abs{\sum_{t=1}^{n}f\left(\frac{x_{t}-d_{n}a}{h_{n}}\right)u_{t}^{(>)}}\neq0\right\}  & \leq\Prob\left\{ \sup_{t\leq n}\smlabs{u_{t}^{(>)}}\neq0\right\} \leq\Prob\left\{ \max_{-n\leq t\leq n}\smlabs{\eta_{t}}>n^{1/\etamom}\right\} =o(1),
\end{align*}
where the final equality follows by Theorem~2.12.1 in \citet{Hansen12mimeo},
since $\{\eta_{t}\}$ is i.i.d.\ with bounded $\etamom$th moment.
\eqref{covorderest} will therefore follow once we have shown that
\begin{equation}
\frac{1}{(\nseq_{n}h_{n})^{1/2}}\sup_{a\in\reals}\abs{\sum_{t=1}^{n}f\left(\frac{x_{t}-d_{n}a}{h_{n}}\right)u_{t}^{(\leq)}}\isbigop(1+n^{1/\etamom-\bandmax})\log n.\label{eq:covorderest2}
\end{equation}

\subsection{Reduction to the maximum over a finite set\label{sub:finitemax}}

For the remainder of the proof, we may without loss of generality
take $f$ to be bounded by unity, with a Lipschitz constant of unity.
To simplify the exposition, we shall require that $h_{n}\in\Hspc_{n}$
always, and take $\abv h=1$; the proof in the general case (where
this occurs w.p.a.1) requires no new ideas. As it is less cumbersome
to work with the inverse bandwidth $b\defeq h^{-1}$, we define 
\[
\Bspc_{n}\defeq\{h^{-1}\mid h\in\Hspc_{n}\}=[1,\abv b_{n}]
\]
where $\abv b_{n}\defeq\blw h_{n}^{-1}$. For $(a,b)\in\reals\times\reals_{+}$,
let $f_{(a,b)}(x)\defeq b^{1/2}f[b(x-d_{n}a)]$. Then
\[
\frac{1}{(\nseq_{n}h)^{1/2}}\sum_{t=1}^{n}f\left(\frac{x_{t}-d_{n}a}{h}\right)u_{t}^{(\leq)}=\frac{1}{\nseq_{n}^{1/2}}\sum_{t=1}^{n}f_{(a,b)}(x_{t})u_{t}^{(\leq)}\eqdef\R_{n}^{f}(a,b)
\]
for $b=h^{-1}$.

Take $C_{n}\defeq[-n^{\gamma},n^{\gamma}]\times\Bspc_{n}$, and let
$\Cspc_{n}\subset C_{n}$ be a lattice of mesh $n^{-\delta}$. $p_{n}(a,b)$
denotes the projection of $(a,b)$ onto a nearest neighbour in $\Cspc_{n}$
(with some tie-breaking rule). We shall now prove that $\gamma$ and
$\delta$ may be chosen (sufficiently large) such that
\begin{gather}
\sup_{(a,b)\in C_{n}}\smlabs{\R_{n}^{f}(a,b)}=\sup_{(a,b)\in\Cspc_{n}}\smlabs{\R_{n}^{f}(a,b)}+o_{p}(1)\label{eq:reduct1}\\
\sup_{(a,b)\in[-n^{\gamma},n^{\gamma}]^{c}\times\Bspc_{n}}\smlabs{\R_{n}^{f}(a,b)}=o_{p}(1),\label{eq:reduct2}
\end{gather}
with the aid of the following.
\begin{lem}
\label{lem:gridapprox}For every $\gamma>0$, there exists a $\delta>0$
such that
\[
\sup_{(a,b)\in C_{n}}\frac{1}{\nseq_{n}^{1/2}}\sum_{t=1}^{n}\smlabs{f_{(a,b)}(x_{t})-f_{p_{n}(a,b)}(x_{t})}\smlabs{u_{t}^{(\leq)}}=o_{p}(1).
\]

\end{lem}
 
\begin{lem}
\label{lem:liptail}Suppose $f\in\BIl$. Then $\smlabs{f(x)}=o(\smlabs x^{-1/2})$
as $x\goesto\pm\infty$.
\end{lem}

Since $\smlnorm{u_{t}^{(\leq)}}_{\infty}\lesssim n^{1/\etamom}$,
\lemref{gridapprox} may be proved by an argument identical to that
used in the proof of \XREFgridapprox{} in \citet{Duffy14uniform},
while \lemref{liptail} is a special case of \XREFliptail{} in the
Supplement to that paper. Observe that \eqref{reduct1} follows immediately
from \lemref{gridapprox}. To establish \eqref{reduct2}, first note
that w.p.a.1,
\[
\inf_{t\leq n}\inf_{\smlabs a\geq n^{\gamma}}\smlabs{x_{t}-d_{n}a}\geq d_{n}n^{\gamma}\left(1-n^{-\gamma}d_{n}^{-1}\max_{t\leq n}\smlabs{x_{t}}\right)=d_{n}n^{\gamma}(1+o_{p}(1)),
\]
provided that $\gamma$ is chosen large enough that
\[
\expect\max_{t\leq n}\smlabs{x_{t}}\leq n^{2}\expect\smlabs{v_{0}}=o(n^{\gamma}d_{n}).
\]
For the proof that such a $\gamma$ exists, see the arguments following
\XREFlipobj{} in the Supplement to \citet{Duffy14uniform}. Thence
by \lemref{liptail}
\begin{align*}
\max_{t\leq n}\sup_{(a,b)\in[-n^{\gamma},n^{\gamma}]^{c}\times\Bspc_{n}}b^{1/2}f[b(x_{t}-d_{n}a)] & \lesssim\max_{t\leq n}\sup_{\smlabs a\geq n^{\gamma}}\smlabs{x_{t}-d_{n}a}^{-1/2}\\
 & \isbigop(d_{n}n^{\gamma})^{-1/2},
\end{align*}
whence
\begin{align*}
\sup_{(a,b)\in[-n^{\gamma},n^{\gamma}]^{c}\times\Bspc_{n}}\smlabs{\R_{n}^{f}(a,b)} & \leq\sup_{(a,b)\in[-n^{\gamma},n^{\gamma}]^{c}\times\Bspc_{n}}\frac{1}{\nseq_{n}^{1/2}}\sum_{t=1}^{n}\smlabs{f_{(a,b)}(x_{t})}\smlabs{u_{t}^{(\leq)}}\\
 & \isbigop\left(\frac{n^{2/\etamom}}{\nseq_{n}d_{n}n^{\gamma}}\right)^{1/2}\\
 & =o(1)
\end{align*}
for $\gamma>0$ sufficiently large.

\subsection{Control over the finite set}

It remains to provide an estimate for
\begin{align*}
\sup_{(a,b)\in\Cspc_{n}}\smlabs{\R_{n}^{f}(a,b)} & =\frac{1}{\nseq_{n}^{1/2}}\sup_{(a,b)\in\Cspc_{n}}\abs{\sum_{t=1}^{n}f_{(a,b)}(x_{t})u_{t}^{(\leq)}}=\frac{1}{\nseq_{n}^{1/2}}\sup_{g\in\Gset_{n}}\smlabs{\S_{n}g},
\end{align*}
where $\Gset_{n}\defeq\{f_{(a,b)}\mid(a,b)\in\Cspc_{n}\}$, and $\S_{n}g$
is defined as in \eqref{Snf} above. Since $f\in\BI$ and $\abv b_{n}\lesssim\nseq_{n}n^{-2\bandmax}$,
it is clear that
\begin{align*}
\delta_{n}(\Gset_{n}) & \lesssim\truncnorm{\eta}\abv b_{n}^{1/2}+[\truncnorm{\eta}+\nseq_{n}^{1/2}]=O[\nseq_{n}^{1/2}(1+n^{1/\etamom-\bandmax})]
\end{align*}
and thus
\[
\frac{1}{\nseq_{n}^{1/2}}\sup_{g\in\Gset_{n}}\smlabs{\S_{n}g}\isbigop(1+n^{1/\etamom-\bandmax})\log n
\]
by \propref{covtruncord}. This completes the proof of \eqref{covorderest}.

\subsection{Modifications required for the proof of \eqref{0enorderest}}

The proof of \eqref{0enorderest} is almost identical to the preceding,
albeit somewhat simpler. The truncation performed in \subref{truncation}
is not necessary, while the same argument as given in \subref{finitemax}
may be used to reduce the problem to that of providing a suitable
bound for
\[
\frac{1}{\nseq_{n}^{1/2}}\sup_{g\in\Gset_{n}}\smlabs{\S_{n}^{\ast}g}
\]
where $\S_{n}^{\ast}g\defeq e_{n}^{-1/2}\sum_{t=1}^{n}g(x_{t})$.
Define
\begin{align*}
\delta_{n}^{\ast}(\Fset) & \defeq\smlnorm{\Fset}_{\infty}+\nseq_{n}^{1/2}(\smlnorm{\Fset}_{1}+\smlnorm{\Fset}_{2})+\left[\sum_{k=1}^{n}d_{k}^{-2}+\nseq_{n}^{1/2}\sum_{k=1}^{n-1}k^{-1/2}d_{k}^{-3/2}\right]\smlalph{\Fset}1\\
 & \lesssim\smlnorm{\Fset}_{\infty}+\nseq_{n}^{1/2}(\smlnorm{\Fset}_{1}+\smlnorm{\Fset}_{2}+\smlalph{\Fset}1)
\end{align*}
where $\smlalph f1\defeq\inf\{c\in\reals_{+}\mid\smlabs{\hat{f}(\lambda)}\leq c\smlabs{\lambda}\}$,
and the final bound follows from parts~\enuref{dm2} and \enuref{dna}
of \lemref{summable}. With the aid of \XREFalphzeroen{} in \citet{Duffy14uniform},
it is easily verified that
\begin{align*}
\smlnorm{\Gset_{n}}_{1}\pmax\smlnorm{\Gset_{n}}_{2}\pmax\smlalph{\Gset_{n}}1 & \lesssim1 & \smlnorm{\Gset_{n}}_{\infty} & \lesssim\abv b_{n}^{1/2}=o(\nseq_{n}^{1/2}),
\end{align*}
whence by \XREFmaxSnb{} in \citet{Duffy14uniform}, 
\[
\frac{1}{\nseq_{n}^{1/2}}\sup_{g\in\Gset_{n}}\smlabs{\S_{n}^{\ast}g}\isbigop\nseq_{n}^{-1/2}\delta_{n}^{\ast}(\Gset_{n})\log n\lesssim\log n
\]
as required.

\section{Proofs of the convergence rates\label{sec:rateproof}}

Recall that 
\[
\locest(a)\defeq\frac{1}{\nseq_{n}}\sum_{t=1}^{n}K_{h_{n}}(x_{t}-d_{n}a),
\]
where $K$ satisfies \assref{estim}.
\begin{proof}[Proof of \propref{Raydenom}]
 Define 
\begin{align*}
\minpr_{n} & \defeq d_{n}^{-1}x_{(1)} & \minpr & \defeq\inf_{r\in[0,1]}\smlabs{X(r)}\\
\maxpr_{n} & \defeq d_{n}^{-1}x_{(n)} & \maxpr & \defeq\sup_{r\in[0,1]}\smlabs{X(r)}.
\end{align*}
Since $(\alpha,H)=(2,\tfrac{1}{2})$, $X$ is a Brownian motion, and
so by \citeauthor{Ray1963IJM}'s \citeyearpar{Ray1963IJM} theorem,
\begin{equation}
\Prob\left\{ \inf_{a\in[\minpr,\maxpr]_{\varepsilon}}\loctime(a)>0\right\} =1.\label{eq:Ray}
\end{equation}
For a more detailed argument as to why \eqref{Ray} follows from Ray's
theorem, see (6.4.36)--(6.4.38) and the surrounding discussion in
\citet{KS91}. (Note that it is necessary that $\varepsilon>0$ here,
since $\loctime(\minpr)=\loctime(\maxpr)=0$ by the continuity of
$\loctime$.) Under the assumption that $\expect\epsilon_{0}^{2}$,
$X_{n}\wkc X$ on $\ell_{\infty}[0,1]$ by \citet{Hannan1979SP},
whence $(\minpr_{n},\maxpr_{n})\wkc(\minpr,\maxpr)$ by the continuous
mapping theorem (CMT). By \XREFgeneralIP{} in \citet{Duffy14uniform},
$\locest\wkc\loctime$ on $\ell_{\infty}(\reals)$, and thus a further
application of the CMT yields
\begin{equation}
\inf_{x\in R_{n}^{\varepsilon}}\frac{1}{\nseq_{n}}\sum_{t=1}^{n}K_{h_{n}}(x_{t}-x)=\inf_{a\in[\minpr_{n},\maxpr_{n}]_{\varepsilon}}\locest(a)\wkc\inf_{a\in[\minpr,\maxpr]_{\varepsilon}}\loctime(a).\label{eq:infloc}
\end{equation}
Together, \eqref{Ray} and \eqref{infloc} yield \eqref{denomord}.

To obtain \eqref{mostsupp}, we note that
\begin{equation}
\frac{1}{n}\sum_{t=1}^{n}\indic\{x_{t}>(1-\varepsilon)x_{(n)}\}\leq\int_{0}^{1}\indic\{X_{n}(r)\geq(1-\varepsilon)\maxpr_{n}\}\diff r+o_{p}(1).\label{eq:xoutofrange}
\end{equation}
Let $\{g_{k}\}$ denote a uniformly bounded sequence of functions
that converges pointwise to $x\elmap\indic\{x\geq0\}$, from above.
Then
\begin{align}
\int_{0}^{1}\indic\{X_{n}(r)\geq(1-\varepsilon)\maxpr_{n}\}\diff r & \leq\int_{0}^{1}g_{k}[X_{n}(r)-(1-\varepsilon)\maxpr_{n}]\diff r\nonumber \\
 & \wkc\int_{0}^{1}g_{k}[X(r)-(1-\varepsilon)\maxpr]\diff r\nonumber \\
 & \inas\int_{0}^{1}\indic\{X(r)\geq(1-\varepsilon)\maxpr\}\diff r,\label{eq:outrange1}
\end{align}
as $n\goesto\infty$ and then $k\goesto\infty$, by the CMT and the
dominated convergence theorem. Further,
\begin{align}
\int_{0}^{1}\indic\{X(r) & \geq(1-\varepsilon)\maxpr\}\diff r\inas\int_{0}^{1}\indic\{X(r)=\maxpr\}\diff r=\int_{\reals}\indic\{x=\maxpr\}\loctime(x)\diff x=0\label{eq:outrange2}
\end{align}
as $\varepsilon\goesto0$, by dominated convergence, \eqref{otf},
and the fact that $\loctime(\maxpr)=0$. It follows from \eqref{xoutofrange}--\eqref{outrange2}
that $\varepsilon>0$ may be chosen such that\noeqref{eq:outrange1}
\[
\limsup_{n\goesto\infty}\Prob\left\{ \frac{1}{n}\sum_{t=1}^{n}\indic\{x_{t}>(1-\varepsilon)x_{(n)}\}\geq\delta\right\} \leq\frac{\delta}{2}.
\]
By an analogous argument, this holds also when $\indic\{x_{t}>(1-\varepsilon)x_{(n)}\}$
is replaced by $\indic\{x_{t}<(1-\varepsilon)x_{(1)}\}$.
\end{proof}

The proof of \thmref{rates} requires the following two results. For
a matrix $A$, let $\smlnorm A_{T}\defeq\sup_{\smlnorm x=1}\smlnorm{Ax}_{2}$.
\begin{lem}
\label{lem:ltapprox} For every $g\in\BIl$ with $\int\smlabs{g(x)x}\diff x<\infty$,
\begin{equation}
\sup_{a\in\reals}\abs{\frac{1}{\nseq_{n}h_{n}}\sum_{t=1}^{n}g\left(\frac{x_{t}-d_{n}a}{h_{n}}\right)-\locest(a)\int g}=o_{p}(1).\label{eq:ltapptarget}
\end{equation}

\end{lem}
 
\begin{lem}
\label{lem:suppcvg} Suppose that $Y_{n}(a)$ is a $(k\times k)$
matrix-valued process, such that $Y_{n}(a)$ is positive semi-definite
for every $a\in\reals$ and $n\in\naturals$, and let $\Gamma$ be
a positive definite $(k\times k)$ matrix, for which $\sup_{a\in\reals}\smlnorm{Y_{n}(a)-\Gamma\locest(a)}_{T}=o_{p}(1)$.
Then
\[
\sup_{\{a\mid\locest(a)\geq\varepsilon\}}\smlnorm{Y_{n}(a)^{-1}}_{T}\isbigop1.
\]

\end{lem}
 
\begin{proof}[Proof of \lemref{ltapprox}]
 Setting $f(x)\defeq g(x)-K(x)\int g$, the left side of \eqref{ltapptarget}
may be written as
\[
\sup_{a\in\reals}\abs{\frac{1}{\nseq_{n}h_{n}}\sum_{t=1}^{n}f\left(\frac{x_{t}-d_{n}a}{h_{n}}\right)}\isbigop\frac{\log n}{(\nseq_{n}h_{n})^{1/2}}=o_{p}(1)
\]
by \thmref{orderest}.
\end{proof}
 
\begin{proof}[Proof of \lemref{suppcvg}]
 By multiplying both $Y_{n}$ and $A\locest$ by $A^{-1}$, we may
reduce the problem to one in which $A=I_{k}$. We may also replace
$(Y_{n},\locest)$ by a distributionally equivalent sequence for which
$(Y_{n},\locest)\inas(\loctime,\loctime)$ in $\ell_{\ucc}(\reals^{2})$:
see Theorem~1.10.3 in \citet{VVW96}. Define $r_{n}^{\omega}(a)\defeq\smlnorm{Y_{n}^{\omega}(a)-I_{k}\loctime_{n}^{\omega}(a)}_{T}$,
and let $\Omega_{0}\subset\Omega$ denote a set, with $\Prob\Omega_{0}=1$,
on which $\locest^{\omega}\goesto\loctime^{\omega}$ in $\ell_{\ucc}(\reals)$,
and $r_{n}^{\omega}\goesto0$ in $\ell_{\infty}(\reals$).

It is easily verified that, for $B$ a real symmetric matrix, and
$z>0$,
\[
\lambda_{\min}(B)=z+\lambda_{\min}(B-zI)\geq z-\smlnorm{B-zI}_{T}
\]
where $\lambda_{\min}(B)$ denotes the smallest eigenvalue of $B$.
Thus, fixing an $\omega\in\Omega_{0}$, 
\begin{align*}
\inf_{\{a\mid\locest^{\omega}(a)\geq\varepsilon\}}\lambda_{\min}[Y_{n}^{\omega}(a)] & \geq\inf_{a\in\reals}\left[\locest^{\omega}(a)-r_{n}^{\omega}(a)\right]\indic\{\locest^{\omega}(a)\geq\varepsilon\}\\
 & \geq\varepsilon-\sup_{a\in\reals}r_{n}^{\omega}(a)\\
 & \goesto\varepsilon
\end{align*}
whence
\[
\sup_{\{a\mid\locest^{\omega}(a)\geq\varepsilon\}}\smlnorm{Y_{n}^{\omega}(a)^{-1}}_{T}=\left(\inf_{\{a\mid\locest^{\omega}(a)\geq\varepsilon\}}\lambda_{\min}[Y_{n}^{\omega}(a)]\right)^{-1}\goesto\varepsilon^{-1},
\]
from which the result follows.
\end{proof}
 
\begin{proof}[Proof of \thmref{rates}]
 Since \eqref{ndrate} follows from arguments given in the course
of \subref{rates}, we provide the proof of only \eqref{linrate}
here. In the notation of \citet[pp.~58f.]{FG96}, $\hat{m}_{L}(x)$
is given by the first element of
\[
\hat{\beta}(x)\defeq(X^{\prime}WX)^{-1}X^{\prime}Wy,
\]
which admits the decomposition
\[
\hat{\beta}(x)-\beta(x)=(X^{\prime}WX)^{-1}X^{\prime}W[\vec{m}_{0}-X^{\prime}\beta]+(X^{\prime}WX)^{-1}X^{\prime}Wu,
\]
where $y=(y_{1},\ldots,y_{n})^{\prime}$, $u=(u_{1},\ldots,u_{n})^{\prime}$,
\begin{align*}
X & \defeq\begin{bmatrix}1 & x_{1}-x\\
1 & x_{2}-x\\
\vdots & \vdots\\
1 & x_{n}-x
\end{bmatrix} & W & \defeq\diag\begin{bmatrix}K_{h_{n}}(x_{1}-x)\\
K_{h_{n}}(x_{2}-x)\\
\vdots\\
K_{h_{n}}(x_{n}-x)
\end{bmatrix} & \vec{m}_{0} & \defeq\begin{bmatrix}m_{0}(x_{1})\\
m_{0}(x_{2})\\
\vdots\\
m_{0}(x_{n})
\end{bmatrix}
\end{align*}
and $\beta(x)=(m_{0}(x),m_{0}^{\prime}(x))^{\prime}$.

Observe $X^{\prime}WX$ is a $(2\times2)$ matrix with $(i,j)$th
element
\[
(X^{\prime}WX)_{ij}(x)=h_{n}^{(i+j-2)}\sum_{t=1}^{n}K_{h_{n}}^{[i+j-2]}(x_{t}-x).
\]
Note that by \lemref{ltapprox},
\[
\frac{1}{\nseq_{n}}\sum_{t=1}^{n}K_{h_{n}}^{[i+j-2]}(x_{t}-d_{n}a)-\locest(a)\int K^{[i+j-2]}\inprob0
\]
in $\ell_{\infty}(\reals)$. Hence, for $D\defeq\diag[1,h_{n}]$,
\begin{gather*}
\nseq_{n}^{-1}[D^{-1}(X^{\prime}WX)D^{-1}](d_{n}a)-\major K\loctime(a)\inprob0
\end{gather*}
in $\ell_{\infty}(\reals)$, where $\major K\defeq[\int K^{[i+j-2]}]$
is positive definite. Thus by \lemref{suppcvg},
\begin{align}
\sup_{x\in A_{n}^{\varepsilon}}\smlnorm{[D(X^{\prime}WX)^{-1}D](x)}_{T} & =\sup_{\{a\mid\locest(a)\geq\varepsilon\}}\smlnorm{[D(X^{\prime}WX)^{-1}D](d_{n}a)}_{T}\isbigop\nseq_{n}^{-1}.\label{eq:denom}
\end{align}

To handle the bias term, note that by a Taylor series expansion 
\begin{align*}
\smlabs{m_{0}(x_{t})-\beta_{0}(x)-\beta_{1}(x)(x_{t}-x)} & \leq m^{\prime\prime}(\tilde{x}_{t})\smlabs{x_{t}-x}^{2}
\end{align*}
for all $x\in A_{n}^{\varepsilon}$, where $\tilde{x}_{t}\in[x,x_{t}]$.
Hence for $i\in\{1,2\}$,
\[
\smlabs{\{D^{-1}X^{\prime}W[\vec{m}_{0}-X^{\prime}\beta]\}_{i}(x)}\leq h_{n}^{2}\derivbnd 2(\tilde{A}_{n}^{\varepsilon})\sum_{t=1}^{n}\smlabs{K_{h_{n}}^{[1+i]}(x_{t}-x)},
\]
whence by \propref{generalIP},
\begin{align}
\frac{1}{\nseq_{n}}\sup_{x\in A_{n}^{\varepsilon}}\smlabs{\{D^{-1}X^{\prime}W[\vec{m}_{0}-X^{\prime}\beta]\}_{i}(x)} & \leq\frac{h_{n}^{2}\derivbnd 2(\tilde{A}_{n}^{\varepsilon})}{\nseq_{n}}\sup_{x\in\reals}\sum_{t=1}^{n}\smlabs{K_{h_{n}}^{[1+i]}(x_{t}-x)}\isbigop h_{n}^{2}\derivbnd 2(\tilde{A}_{n}^{\varepsilon}).\label{eq:linbias}
\end{align}
For the variance term, note that for $i\in\{1,2\}$, 
\begin{align}
\frac{1}{\nseq_{n}}\sup_{x\in A_{n}^{\varepsilon}}\smlabs{[D^{-1}X^{\prime}u]_{i}(x)} & =\frac{1}{\nseq_{n}}\sup_{x\in A_{n}^{\varepsilon}}\abs{\sum_{t=1}^{n}K_{h_{n}}^{[i-1]}(x_{t}-x)u_{t}}\isbigop\frac{\log n}{(\nseq_{n}h_{n})^{1/2}}\label{eq:lincov}
\end{align}
by \thmref{orderest}. \eqref{denom}--\eqref{lincov} now yield the
stated result.\noeqref{eq:linbias}
\end{proof}
\bibliographystyle{econometrica}
\bibliography{time-series}

\cleartooddpage{}

\setcounter{page}{1}

\renewcommand{\thepage}{S\arabic{page}}

\appendix

\section{Supplementary material}

\subsection{Proofs of Lemmas~\ref{lem:fourinv}--\ref{lem:summable}\label{app:auxproofs}}
\begin{proof}[Proof of \lemref{fourinv}]
 Let $g_{k}(z):=g(z)\indic\{\smlabs{g(z)}\leq k\}$. $g_{k}$ is
bounded, and a straightforward extension of the argument used to verify
\XREFfourinv{} in \citet{Duffy14uniform} gives that
\[
\expect f(Y)g_{k}(Z)=\frac{1}{2\pi}\int\hat{f}(\lambda)\expect\left[\e^{-\i\lambda^{\prime}Y}g_{k}(Z)\right]\diff\lambda
\]
for every $k\in\naturals$. Now let $k\goesto\infty$; the left side
converges to $\expect f(Y)g(Z)$ by dominated convergence. For the
right side, using that $Y_{1}$ and $(Y_{2},Z)$ are independent,
we have
\begin{align*}
\abs{\int\hat{f}(\lambda)\expect\left[\e^{-\i\lambda^{\prime}Y}\{g_{k}(Z)-g(Z)\}\right]\diff\lambda} & \leq\left(\int\smlabs{\hat{f}(\lambda)\psi_{Y_{1}}(-\lambda)}\diff\lambda\right)\expect\smlabs{g_{k}(Z)-g(Z)}\\
 & \leq\smlnorm f_{1}\smlnorm{\psi_{Y_{1}}}_{1}\expect\smlabs{g(Z)}\indic\{\smlabs{g(Z)}>k\}\\
 & \goesto0
\end{align*}
using the fact that $\smlabs{\hat{f}(\lambda)}\leq\smlnorm f_{1}$.
\end{proof}
 
\begin{proof}[Proof of \lemref{covbnd}]
 We shall give only the proof of \eqref{cov1b} here; the proof of
\eqref{cov1a} follows by similar arguments, and is somewhat simpler.
Recall from \eqref{decomp2} the decomposition 
\begin{align*}
x_{t+1,t+k,t+k}^{\prime} & =a_{m}\epsilon_{t+k-m}+\sum_{\substack{l=0\\
l\neq m
}
}^{k-1}a_{l}\epsilon_{t+k-l}.
\end{align*}
Let $\mathcal{K}\defeq\{\smlfloor{k/2}+1,\ldots,k-1\}\backslash\{m\}$.
Since the second term on the right is independent of $\eta_{t+k-m}$,
\begin{align*}
\smlabs{\expect\eta_{t+k-m}\e^{-\i\lambda x_{t+1,t+k,t+k}^{\prime}}} & \leq\smlabs{\expect\eta_{t+k-m}\e^{-\i\lambda a_{m}\epsilon_{t+k-m}}}\prod_{l\in\mathcal{K}}\smlabs{\psi(-\lambda a_{l})}\\
 & \leq[\smlabs{a_{m}}\smlabs{\lambda}\expect\smlabs{\eta_{0}\epsilon_{0}}\pmin\expect\smlabs{\eta_{0}}]\prod_{l\in\mathcal{K}}\smlabs{\psi(-\lambda a_{l})}\\
 & \lesssim(c_{m}\smlabs{\lambda}\pmin1)\prod_{l\in\mathcal{K}}\smlabs{\psi(-\lambda a_{l})}
\end{align*}
using $\expect\smlabs{\e^{\i x}-1}\leq\smlabs x$, \eqref{coefcf}
and the Cauchy-Schwarz inequality. Hence
\[
\smlabs{\expect\eta_{t+k-m}\e^{-\i\lambda x_{t+1,t+k,t+k}^{\prime}}}^{q}\lesssim(c_{m}^{q}\smlabs{\lambda}^{q}\pmin1)\prod_{l\in\mathcal{K}}\smlabs{\psi(-\lambda a_{l})}.
\]
Thus the left side of \eqref{cov1b} may be bounded above by a constant
times
\[
\int_{\reals}(z_{1}c_{m}^{q}\smlabs{a_{k}}^{p}\smlabs{\lambda}^{p+q}F(a_{k}\lambda)\pmin z_{2})\prod_{l\in\mathcal{K}}\smlabs{\psi(-\lambda a_{l})}\diff\lambda.
\]
The result now follows by \XREFcfprodbnd{} in the Supplement to \citet{Duffy14uniform}.
\end{proof}
 
\begin{proof}[Proof of \lemref{1stmom}]
 \enuref{1stmom} follows by arguments analogous to those used in
the proof of \XREFstmom{} in \citet{Duffy14uniform}. For \enuref{1stcov},
we recall from \eqref{decomp1} the decomposition
\begin{align}
x_{t+k} & =x_{t,t+k}^{\ast}+x_{t+1,t+k,t+k}^{\prime}.\label{eq:smpldecmp}
\end{align}
Thence by Fourier inversion (\lemref{fourinv}) and \lemref{covbnd}\enuref{cov1},
\begin{align*}
\smlabs{\expect_{t}f(x_{t+k})\eta_{t+k-m}} & =\biggabs{\frac{1}{2\pi}\int_{\reals}\hat{f}(\lambda)\e^{-\i\lambda x_{t,t+k}^{\ast}}\expect[\eta_{t+k-m}\e^{-\i\lambda x_{t+1,t+k,t+k}^{\prime}}]\diff\lambda}\\
 & \lesssim\smlnorm f_{1}\int_{\reals}\smlabs{\expect\eta_{t+k-m}\e^{-\i\lambda x_{t+1,t+k,t+k}^{\prime}}}\diff\lambda,
\end{align*}
using the fact that $\smlabs{\hat{f}(\lambda)}\leq\smlnorm f_{1}$.
The result now follows by \lemref{covbnd}\enuref{cov1}.
\end{proof}
 
\begin{proof}[Proof of \lemref{summable}]
 For \enuref{dm2}, note that $\{d_{t}^{-2}\}$ is regularly varying
with index $-2H$, whence by Karamata's theorem and Proposition~1.5.9a
in \citet{Bingham87}, $\{\sum_{t=1}^{n}d_{t}^{-2}\}$ is either slowly
varying (when $H\leq1/2$), or regularly varying with index $1-2H$.
In comparison, $\{\nseq_{n}^{1/2}\}$ is regularly varying with index
\[
\frac{1}{2}(1-H)>1-2H
\]
for all $H\in(\tfrac{1}{3},1)$; thus \enuref{dm2} holds. \enuref{dna}
follows from the fact that $\{k^{-1/2}d_{k}^{-3/2}\}$ is regularly
varying with index
\[
-\frac{1}{2}-\frac{3}{2}H<-\frac{1}{2}-\frac{3}{2}\cdot\frac{1}{3}=-1
\]
For \enuref{dth}, note that $\{c_{m}\}$ and $\{m^{1/2}\nseq_{m}\}$
are regularly varying with indices $H-1/\alpha<1$ and
\[
\frac{1}{2}+1-H<\frac{3}{2}-\frac{1}{3}=\frac{7}{6}
\]
respectively. Thus the result follows from \assref{reg}\enuref{dist}.
\end{proof}
 \cleartooddpage{}

\subsection{List of key notation\label{app:notation}}

\newref{ass}{name = Ass.~}

\newref{sub}{name = Sec.~}

\newref{sec}{name = Sec.~}

\newref{rem}{name = Rem.~}

\newref{prop}{name = Prop.~}

\subsubsection*{Greek and Roman symbols}

Listed in (Roman) alphabetical order. Greek symbols are listed according
to their English names: thus $\Omega$, as `omega', appears before
$\xi$, as `xi'.

\noindent %
\begin{longtable}[l]{>{\raggedright}p{0.11\textwidth}>{\raggedright}p{0.69\textwidth}l}
$a_{i}$ & partial sum of $\{\phi_{i}\}$, $a_{i}\defeq\sum_{j=0}^{i}\phi_{j}$
\cellfill & \secref{prelim}\tabularnewline[\doublerulesep]
$A_{n}^{\varepsilon}$ & subset of $\reals$ on which the normalised `signal' exceeds $\epsilon$\cellfill & \eqref{An}\tabularnewline[\doublerulesep]
$\tilde{A}_{n}^{\varepsilon}$ & slight enlargement of $A_{n}^{\varepsilon}$ \cellfill & \eqref{psi1bnd}\tabularnewline[\doublerulesep]
$\alpha$ & index of domain of attraction of $\epsilon_{0}$ \cellfill & \assref{reg}\enuref{iidseq}\tabularnewline[\doublerulesep]
$\BI$ & bounded and integrable functions on $\reals$ \cellfill & \secref{intro}\tabularnewline[\doublerulesep]
$\BIl$ & Lipschitz functions in $\BI$ \cellfill & \subref{model}\tabularnewline[\doublerulesep]
$c_{n}$ & norming sequence \cellfill & \eqref{ck}\tabularnewline[\doublerulesep]
$C$ & generic constant \cellfill & \secref{intro}\tabularnewline[\doublerulesep]
$d_{n}$ & norming sequence used to define $X_{n}$ \cellfill & \eqref{dkek}\tabularnewline[\doublerulesep]
$\delta_{n}(\Fset)$ & appears in \propref{covtruncord} \cellfill & \eqref{delnf}\tabularnewline[\doublerulesep]
$\nseq_{n}$ & norming sequence used to define $\locest^{f}$ \cellfill & \eqref{dkek}\tabularnewline[\doublerulesep]
$\epsilon_{t}$ & i.i.d.\ sequence \cellfill & \assref{reg}\enuref{iidseq}\tabularnewline[\doublerulesep]
$\eta_{t}$ & i.i.d.\ sequence \cellfill & \assref{reg}\enuref{iidseq}\tabularnewline[\doublerulesep]
$\eta_{t}^{(\leq)},\eta_{t}^{(>)}$ & truncated version of $\eta_{t}$ and remainder \cellfill & \eqref{etatrunc}\tabularnewline[\doublerulesep]
$\smlnorm{\eta}_{n}$ & defined as $\smlnorm{\eta}_{n}\defeq\smlnorm{\eta_{0}^{(\leq)}}_{\infty}$
\cellfill & \subref{order}\tabularnewline[\doublerulesep]
$\expect_{t}$ & expectation conditional on $\filt_{-\infty}^{t}$\cellfill & \secref{prelim}\tabularnewline[\doublerulesep]
$\filt_{s}^{t}$ & $\sigma$-field generated by $\{\epsilon_{r}\}_{r=s}^{t}$ \cellfill & \secref{prelim}\tabularnewline[\doublerulesep]
$\Fset,\Fset_{n},\Gset$ & subsets of $\BI$ \cellfill & \subref{order}\tabularnewline[\doublerulesep]
$G$ & specific slowly varying function\cellfill & \eqref{dofattr}\tabularnewline[\doublerulesep]
$h,h_{n}$ & bandwidth parameter (or sequence) \cellfill & \assref{estim}\tabularnewline[\doublerulesep]
$\blw h_{n},\abv h$ & lower and upper bounds defining $\Hspc_{n}$ \cellfill & \assref{estim}\tabularnewline[\doublerulesep]
$H$ & sets the decay rate of $\phi_{k}$ as $k\goesto\infty$ \cellfill & \assref{reg}\enuref{regproc}\tabularnewline[\doublerulesep]
$\Hspc_{n}$ & set of allowable bandwidths \cellfill & \assref{estim}\tabularnewline[\doublerulesep]
$K$ & smoothing kernel \cellfill & \assref{estim}\tabularnewline[\doublerulesep]
$\ell_{\ucc}(Q)$ & bounded on compacta functions on $Q$, with ucc topology \cellfill & \secref{intro}\tabularnewline[\doublerulesep]
$\ell_{\infty}(Q)$ & bounded functions on $Q$, with uniform topology \cellfill & \secref{intro}\tabularnewline[\doublerulesep]
$\loctime$ & local time of $X$ \cellfill & \eqref{otf}\tabularnewline[\doublerulesep]
$\locest^{f}$ & sample estimate of local time \cellfill & \eqref{uniformlaw}\tabularnewline[\doublerulesep]
$m_{0}$ & regression function \cellfill & \eqref{basicreg}\tabularnewline[\doublerulesep]
$\abv m_{i}(A)$ & bounds the $i$th derivative of $m_{0}$ on $A\subseteq\reals$\cellfill & \eqref{mregularity}\tabularnewline[\doublerulesep]
$\hat{m}$ & local level (Nadaraya-Watson) estimate of $m_{0}$ \cellfill & \subref{model}\tabularnewline[\doublerulesep]
$\hat{m}_{L}$ & local linear estimate of $m_{0}$ \cellfill & \subref{model}\tabularnewline[\doublerulesep]
$\M_{nmk}f$ & martingale components in decomposition of $\S_{nm}f$\cellfill & \eqref{Snmdecomp}\tabularnewline[\doublerulesep]
$\N_{nm}f$ & remainder from decomposition of $\S_{nm}f$\cellfill & \eqref{Snmdecomp}\tabularnewline[\doublerulesep]
$\Omega$ & sample space \cellfill & \secref{rateproof}\tabularnewline[\doublerulesep]
$p_{0}$ & chosen such $\psi\in L^{p_{0}}$ \textbf{\cellfill} & \assref{reg}\enuref{iidseq}\tabularnewline[\doublerulesep]
$\phi_{k}$ & coefficients defining the linear process $v_{t}$ \cellfill & \assref{reg}\enuref{regproc}\tabularnewline[\doublerulesep]
$\pi_{k}$ & slowly varying sequence related to $\phi_{k}$ \cellfill & \assref{reg}\enuref{regproc}\tabularnewline[\doublerulesep]
$\psi$ & characteristic function of $\epsilon_{0}$ \cellfill & \assref{reg}\enuref{iidseq}\tabularnewline[\doublerulesep]
$\Psi_{kn}$ & components in decomposition of $\hat{m}$ \cellfill & \eqref{decmp}\tabularnewline[\doublerulesep]
$q_{0}$ & chosen such that $\expect\smlabs{\eta_{0}}^{q_{0}}<\infty$ \cellfill & \assref{reg}\enuref{iidseq}\tabularnewline[\doublerulesep]
$r_{0}$ & used to define order of $\blw h_{n}$ \cellfill & \assref{estim}\tabularnewline[\doublerulesep]
$R_{n}^{\varepsilon}$ & truncated range of $\{x_{t}\}_{t=1}^{n}$ \cellfill & \subref{rates}\tabularnewline[\doublerulesep]
$\varrho_{n}$ & norming sequence \cellfill & \eqref{stablecvg}\tabularnewline[\doublerulesep]
$\S_{n},\S_{nm}$ & covariance summation operator, $\S_{n}f\defeq\sum_{t=1}^{n}f(x_{t})u_{t}^{(\leq)}$
\cellfill & \eqref{Snf}, \eqref{Snmdecomp}\tabularnewline[\doublerulesep]
$\tau_{1}$ & function $x\elmap\e^{x}-1$ \cellfill & \secref{prelim}\tabularnewline[\doublerulesep]
$\theta_{k}$ & coefficients defining the linear process $u_{t}$ \cellfill & \eqref{disturbance}\tabularnewline[\doublerulesep]
$u_{t}$ & regression disturbance; linear process built from $\{\eta_{t}\}$
\cellfill & \eqref{basicreg}, \eqref{disturbance}\tabularnewline[\doublerulesep]
$u_{t}^{(\leq)},u_{t}^{(>)}$ & analogues of $u_{t}$ built from $\{\eta_{t}^{(\leq)}\}$ and $\{\eta_{t}^{(>)}\}$
\cellfill & \eqref{utrunc}\tabularnewline[\doublerulesep]
$v_{t}$ & linear process built from $\{\epsilon_{t}\}$ \cellfill & \eqref{regproc}\tabularnewline[\doublerulesep]
$x_{t}$ & regressor process; partial sum of $\{v_{t}\}$ \cellfill & \eqref{basicreg}, \eqref{regproc}\tabularnewline[\doublerulesep]
$x_{(i)}$ & $i$th order statistic of $\{x_{t}\}_{t=1}^{n}$; $x_{(i)}\leq x_{(i+1)}$
\cellfill & \subref{rates}\tabularnewline[\doublerulesep]
$x_{s,t}^{\ast}$ & $\filt_{-\infty}^{s}$-measurable component of $x_{t}$ \cellfill & \eqref{decomp1}\tabularnewline[\doublerulesep]
$x_{s,r,t}^{\prime}$ & $\filt_{s}^{r}$-measurable component of $x_{t}$ \cellfill & \eqref{decomp2}\tabularnewline[\doublerulesep]
$X$ & finite-dimensional limit of $X_{n}$, an LFSM \cellfill & \eqref{Xdef}\tabularnewline[\doublerulesep]
$X_{n}$ & process constructed from $\{x_{t}\}$ \cellfill & \subref{asympreg}\tabularnewline[\doublerulesep]
$\xi_{mkt}f$ & martingale difference components of $\M_{nmk}f$ \cellfill & \eqref{xi}\tabularnewline[\doublerulesep]
$y_{t}$ & dependent variable in the regression \cellfill & \eqref{basicreg}\tabularnewline[\doublerulesep]
$Z_{\alpha}$ & $\alpha$-stable Lvy motion \cellfill & \eqref{stablecvg}\tabularnewline[\doublerulesep]
\end{longtable}

\newpage{}

\subsubsection*{Symbols not connected to Greek or Roman letters}

Ordered alphabetically by their description.

\noindent %
\begin{longtable}[l]{>{\raggedright}p{0.11\textwidth}>{\raggedright}p{0.69\textwidth}l}
$\eqdist$ & both sides have the same distribution \cellfill & \remref{exog}\tabularnewline[\doublerulesep]
$\smlceil{\cdot}$ & ceiling function \cellfill & \secref{intro}\tabularnewline[\doublerulesep]
$\inprob$ & converges in probability to \cellfill & \subref{rates}\tabularnewline[\doublerulesep]
$\fdd$ & finite-dimensional convergence \cellfill & \secref{intro}\tabularnewline[\doublerulesep]
$\smlfloor{\cdot}$ & floor function (integer part) \cellfill & \secref{intro}\tabularnewline[\doublerulesep]
$\hat{f}$ & Fourier transform of $f$ \cellfill & \secref{prelim}\tabularnewline[\doublerulesep]
$\lesssim$ & left side bounded by a constant times the right side \cellfill & \secref{intro}\tabularnewline[\doublerulesep]
$\isbigop$ & left side bounded in probability by the right side \cellfill & \secref{intro}\tabularnewline
 & ($a_{n}\isbigop b_{n}$ if $a_{n}=O_{p}(b_{n})$) & \tabularnewline[\doublerulesep]
$\smlnorm f_{p}$ & $L^{p}$ norm, $(\int\smlabs f^{p})^{1/p}$, for function $f$ \cellfill & \secref{intro}\tabularnewline
 & denotes $\sup_{x\in\reals}\smlabs{f(x)}$ when $p=\infty$ & \tabularnewline[\doublerulesep]
$\smlnorm X_{p}$ & $L^{p}$ norm, $(\expect\smlabs X^{p})^{1/p}$, for random variable
$X$ \cellfill & \secref{intro}\tabularnewline[\doublerulesep]
$\smlcv M$ & martingale conditional variance \cellfill & \eqref{qvcv}\tabularnewline[\doublerulesep]
$\smlqv M$ & martingale sum of squares \cellfill & \eqref{qvcv}\tabularnewline[\doublerulesep]
$\#\Fset$ & number of elements in the (finite) set $\Fset$ \cellfill & \propref{covtruncord}\tabularnewline[\doublerulesep]
$\smlnorm X_{\tau}$ & Orlicz norm associated to function $\tau$ \cellfill & \secref{prelim}\tabularnewline[\doublerulesep]
$f^{[p]}$ & product $x\elmap x^{p}f(x)$ \cellfill & \eqref{psi1bnd}\tabularnewline[\doublerulesep]
$\sim$ & strong asymptotic equivalence \cellfill & \secref{intro}\tabularnewline
 & ($a_{n}\sim b_{n}$ if $\lim_{n\goesto\infty}a_{n}/b_{n}=1$) & \tabularnewline[\doublerulesep]
$\smlnorm{\Fset}$ & supremum of norm $\smlnorm{\cdot}$ over $\Fset$: $\sup_{f\in\Fset}\smlnorm f$
\cellfill & \subref{order}\tabularnewline[\doublerulesep]
$\asymp$ & weak asymptotic equivalence \cellfill & \secref{intro}\tabularnewline[\doublerulesep]
 & ($a_{n}\asymp b_{n}$ if $\lim_{n\goesto\infty}a_{n}/b_{n}\in(-\infty,\infty)\backslash\{0\}$) & \tabularnewline[\doublerulesep]
$\wkc$ & weak convergence (\citealp{VVW96}) \cellfill & \secref{intro}\tabularnewline[\doublerulesep]
\end{longtable}
\end{document}